\documentclass[12pt]{article}
\usepackage[utf8]{inputenc}
\usepackage[left=2cm, right=2cm,, top=3cm, centering]{geometry}
\usepackage{hyperref}
\usepackage{amsmath}
\usepackage{amsfonts}
\usepackage{amssymb}
\usepackage{amsthm}
\usepackage{mathtools}
\usepackage{mathabx}
\usepackage{graphicx}
\usepackage{enumitem}
\allowdisplaybreaks[4]

\usepackage{hyperref}
\hypersetup{
	colorlinks=true,     
	linkcolor=blue,         
	citecolor=red}
 
\mathtoolsset{showonlyrefs=true,showmanualtags=true} 

\DeclareMathOperator{\var}{Var}

\DeclareMathOperator{\vol}{Vol}
\DeclareMathOperator{\cov}{Cov}

\newcommand{\R}{\mathbb{R}}
\newcommand{\E}{\mathbb{E}}
\newcommand{\C}{\mathbb{C}}
\newcommand{\Z}{\mathbb{Z}}

\newcommand{\B}{\mathcal B}

\renewcommand{\L}{ L}
 
\newcommand{\f}{\varphi}
\renewcommand{\phi}{\varphi}

\newcommand{\os}[2]{\overset{#1}{#2}}

\newcommand{\scp}[1]{\langle #1 \rangle}

\theoremstyle{plain}
\newtheorem{theorem}{Theorem}
\newtheorem*{theorem*}{Theorem}
\newtheorem{proposition}[theorem]{Proposition}
\newtheorem*{proposition*}{Proposition}
\newtheorem{lemma}[theorem]{Lemma}
\newtheorem{corollary}[theorem]{Corollary}

\theoremstyle{definition}

\newtheorem*{definition*}{Definition}
\newtheorem{remark}{Remark}
\newtheorem*{remark*}{Remark}
\newtheorem{example}{Example}

\newtheorem{counterexample}{Counterexample}[section]

\title{Limit theorems for $p$-domain functionals of stationary Gaussian fields}
\author{Nikolai Leonenko\footnote{University of Cardiff \texttt{leonenkoN@cardiff.ac.uk}}, Leonardo Maini\footnote{University of Luxembourg \texttt{leonardo.maini@uni.lu},  \texttt{ivan.nourdin@uni.lu}, \texttt{francesca.pistolato@uni.lu}}, Ivan Nourdin\footnotemark[2], Francesca Pistolato\footnotemark[2]}
\date{February 26, 2024}

\begin{document}

\maketitle

\begin{abstract}
    Fix an integer $p\geq 1$ and refer to it as the number of growing domains. 
    For each $i\in\{1,\ldots,p\}$, fix a compact subset $D_i\subseteq\R^{d_i}$ where 
    $d_1,\ldots,d_p\ge 1$. Let $d= d_1+\dots+d_{p}$ be the total underlying dimension. Consider a continuous, stationary, centered Gaussian field $B=(B_x)_{x\in \R^d}$ with unit variance. Finally, let $\varphi:\R \rightarrow \R$ be a measurable function such that $\E[\varphi(N)^2]<\infty$ for $N\sim N(0,1)$. 
    
    In this paper, we investigate central and non-central limit theorems as $t_1,\ldots,t_p\to\infty$ for functionals of the form
    \[
    Y(t_1,\dots,t_p):=\int_{t_1D_1\times\dots \times t_pD_p}\varphi(B_{x})dx.
    \]

Firstly, we assume that the covariance function $C$ of $B$ is {\it separable} (that is, $C=C_1\otimes\ldots\otimes C_{p}$ with $C_i:\R^{d_i}\to\R$), and thoroughly investigate under what condition 
    $Y(t_1,\dots,t_p)$ satisfies a central or non-central limit theorem when the same holds for $\int_{t_iD_i}\varphi(B^{(i)}_{x_i})dx_i$ for at least one (resp. for all) $i\in \{1,\ldots,p\}$, where
    $B^{(i)}$ stands for a stationary, centered, Gaussian field on $\R^{d_i}$ admitting $C_i$ for covariance function.
    When $\varphi$ is an Hermite polynomial, we also provide a quantitative version of the previous result, which improves some bounds from \cite{RST}. 

Secondly, we extend our study beyond the separable case, examining what can be inferred when the covariance function is either in the Gneiting class or is additively separable.
    
\end{abstract}

\section{Introduction}

Gaussian fields are widely used to model random quantities arising in different applications, see e.g. \cite{LRM23, Diggle2007, G02, Christakos1992RandomFM}. These random quantities may depend for instance on space $x_1\in\R^{d_1}$ and time $x_2\in\R_+$ (e.g. \cite{G02}) or more generally on several variables $x_1,x_2,x_3,\dots$ belonging to (possibly non-Euclidean) spaces of different dimensions (e.g. \cite{porcu1motivation, MRV21MotivNonUni,marinucci2022fluctuations}). 

Throughout this paper, we consider a continuous, stationary, real-valued, centered Gaussian field $B=(B_{x})_{x\in \R^d}$ with unit variance, where $d\geq 2$. We denote by $C:\R^d\rightarrow\R$ the \textbf{covariance function} of $B$, defined as
$$
{\rm Cov}(B_x,B_y)=C(x-y),\quad \quad x,y\in \R^{d}.
$$

We also consider a non constant measurable function $\varphi:\R\rightarrow\R$ satisfying $\E[\varphi^2(N)]<\infty$ for $N\sim N(0,1)$. 
As is well-known, $\varphi$ admits an Hermite decomposition (see, e.g. \cite[Section 1.4]{bluebook}) of the form
\begin{equation}\label{hermitedecomp}
\varphi =\E[\varphi(N)]+ \sum_{q=R}^\infty a_q H_q,\quad \mbox{with $R\geq 1$ such that $a_R\neq 0$},
\end{equation}
where $H_q$ is the $q$th Hermite polynomial and $a_q=a_q(\varphi)=\frac{1}{q!}\mathbb{E}\left[\varphi(N)H_q(N)\right]$.
The integer $R\geq 1$ is called the {\bf Hermite rank} of $\varphi$.

Finally, we consider a family of compact subsets $D_i\subseteq\R^{d_i}$, $1\leq i\leq p$, satisfying $\vol(D_i)>0$ for each $i$. The number $d= d_1+\dots+d_{p}$ is called the total dimension.

The main object of interest of our paper is the additive functional 
\begin{equation}
    \label{func}
    Y(t_1,\dots,t_p):=\int_{ t_1D_1\times\dots \times t_pD_p}\varphi(B_{x})dx\quad\mbox{for $t_1,\dots,t_p>0$}.
\end{equation}
We remark that the integral in~\eqref{func} is well defined thanks to the continuity assumption made on $B$ as well as the square integrability of $\varphi$ with respect to the standard Gaussian measure (see e.g. \cite[Proposition 3]{MN22}).

Under sufficient conditions ensuring that $\var(Y(t_1,\ldots,t_p))>0$ for every $t_1,\ldots,t_p>0$ large enough, we study the limit in distribution as $t_1,\ldots,t_p\to\infty$ of the normalized version of the functional $Y(t_1,\ldots,t_p)$ defined as
\begin{equation}
    \label{mainquestion}
    \widetilde{Y}(t_1,\ldots,t_p):=\frac{Y({t_1,\ldots,t_p})-\E[Y({t_1,\ldots,t_p})]}{\sqrt{{\rm Var}(Y(t_1,\ldots,t_p))}}.
\end{equation}

While the case $p=1$ has been extensively studied since the eighties (see e.g. the seminal works \cite{BM83}, \cite{DM79}, \cite{R60}, \cite{T79}), not much literature can be found when the integration domain of (\ref{func}) does not grow uniformly with respect to all its directions (that is, when $p\geq 2$ and when $t_1,\ldots,t_p$ go to infinity at possibly different rates). In fact, the extended class of functionals \eqref{func} has aroused interest only in the past decade (see e.g. \cite{RST}, \cite{PR16}, \cite{B11}, \cite{LRM23}, \cite{AO17}, \cite{PS17}), because of the more and more important role that \textbf{spatio-temporal} functionals of random fields or random fields with {\bf separable covariance function} (e.g. the fractional Brownian sheet) play now in applications. 

In our work, the number $p$ of growing domains and their dimensions $d_i$ can be arbitrary. This is why we refer to \eqref{func} as a \textbf{$p$-domain} (rather than a {\it spatio-temporal}) {\bf functional}. To give at least one explicit motivation for studying the asymptotic behavior of this type of functionals, let us consider the case where $\varphi=\mathbf{1}_{[a,\infty)}$. It corresponds to the excursion volume of $B$ at level $a\in\R$ in the observation window $t_1D_1\times\dots \times t_pD_p$, namely
\[
Y(t_1,\dots,t_p)=\vol\left(\left\{\right(x_1,\dots,x_{p})\in t_1D_1\times\ldots\times t_pD_p\,:\,B_{(x_1,\dots,x_{p})}\ge a\}\right),
\]
which is a (random) geometrical object that has been extensively studied in the literature, see e.g \cite{MR2319516, LRM23}. 
Our approach offers more flexibility than the usual one (corresponding to $p=1$), as it gives the possibility to study the excursion sets of $B$ when the parameters are in domains $t_1D_1,\ldots,t_pD_p$ that can grow at different rates.

In a broader setting, the aim of the present work is to enhance our understanding of the asymptotic behaviour of $p$-domain functionals associated with stationary Gaussian fields, when the covariance function has a specific form. In particular, we will show how, in some cases, this asymptotic behaviour can be simply obtained from that of $1$-domain functionals, explaining in a new light and in a more systematic way some results from the recent literature (e.g. those contained in \cite{RST}). 

\subsection{Separable case}
\label{subsec:mainres}
In this section, we investigate the asymptotic behavior of (\ref{mainquestion}) when we assume that the covariance function of $B$ is separable.

A covariance function $C:\R^d\to\R$ is said \textbf{separable} when it can be written as $C=C_1\otimes\ldots\otimes C_p$, that is, when
\begin{equation}
    \label{equ:sep}
    C(x_1,\ldots,x_p)= \prod_{i=1}^p C_i(x_i),\quad x_i\in \R^{d_i},\quad 1\leq i\leq p,
\end{equation}
for some functions $C_i:\R^{d_i}\to\R$ satisfying $C_i(0)=1$ for each $i$.  
It is easy to check that $C$ is non-negative definite if and only if $C_i$ is non-negative definite for every $i$. Also, by stationarity of $B$ on $\R^d$, we have that $C_i$ is the covariance function of the field $B^{(i)}:=(B_{(x_1,\ldots,x_i,\ldots,x_p)})_{x_i\in\R^{d_i}}$ for any fixed values of $x_j$, $j\neq i$. Since we are only interested in distributions, we can define the \textbf{marginal functionals} 
\begin{equation} \label{marginal funct}   
Y_i(t_i):=\int_{t_iD_i}\varphi(B_{x_i}^{(i)})dx_i, \quad i=1,\ldots p,
\end{equation}
and their normalized versions
\begin{equation}\label{stand marginal funct}
    \widetilde{Y}_i(t_i):=\frac{Y_i(t_i)-\E[Y_i(t_i)]}{\sqrt{{\rm Var}(Y_i(t_i))}},\quad i=1,\ldots p,
\end{equation}
with the convention that, for any given $i$, the values of $x_j$, $j\neq i$, in (\ref{marginal funct}) are arbitrary but fixed.

Let us denote by $\overset{d}{\rightarrow}$ the convergence in distribution. The main result of this section is the following. 
\begin{theorem}\label{thm:main}
Let $\f:\R \rightarrow \R$ be a measurable function satisfying $\E[\varphi^2(N)]<\infty$ for $N\sim N(0,1)$, with Hermite rank $R\geq 1$. 
    Let $B=(B_x)_{x\in\R^d}$ be a real-valued, continuous, centered, stationary Gaussian field with unit-variance. 
    Let $C:\R^d \rightarrow \R$ be the covariance function of $B$, assume it is separable in the sense of \eqref{equ:sep}, and also that it satisfies, for each $i$:
$$
C^R\ge0\quad\mbox{and}\quad C_i\in \bigcup_{M=1}^\infty L^M(\R^{d_i}).
$$
   Let us consider $\widetilde Y$ given by \eqref{mainquestion} and $\widetilde Y_i$ given in \eqref{stand marginal funct}. 
    Then, the following two assertions are equivalent:
    \begin{enumerate}
        \item[(a)]
        $\widetilde{Y}_{i}(t_i)\overset{d}{\rightarrow} N(0,1)$ as $t_i \to \infty$ for at least one $i\in\{1,\ldots,p\}$;   
        \item[(b)]
        $\widetilde{Y}(t_1,\ldots,t_p)\overset{d}{\rightarrow} N(0,1)$ as $t_1,\ldots,t_p \to \infty$.
        \end{enumerate}
    
\end{theorem} 

\begin{remark}
    The integrability assumptions on $C_i$ for Theorem \ref{thm:main} to hold may be removed when $\varphi=H_q$; moreover, in this case we also have \textbf{quantitative results} for the convergence in distribution (see Subsection \ref{subsec:Hq}).
\end{remark}

\begin{remark}
    Given a {\it discrete} stationary centered Gaussian field $B=(B_k)_{k\in \mathbb Z^d}$ with unit variance, the definition of separable covariance function $C:\mathbb Z^d \to \R$ is similar, i.e. $
        C(z) = \prod_{i=1}^p C_i(z_i)$ with $C_i:\Z^{d_i}\to \R$ for $i=1,\ldots,p$, whereas the functionals \eqref{func} and \eqref{marginal funct} are defined as follows:
    \[
    Y(n_1,\dots,n_p)=\sum_{k\in (\prod_{i=1}^p[0,n_i]^{d_i})\cap \mathbb Z^{d}}\varphi(B_{k}), \quad \quad Y_i(n_i)=\sum_{k_i\in  [0,n_i]^{d_i}\cap \mathbb Z^{d_i}}\varphi(B_{k_1,\dots,k_p}),
    \]
    where $n_i\in \mathbb N$ for every $i=1,\ldots,p$. In this setting, analogous results to Theorem~\ref{thm:main} could be obtained, see e.g. Remark~\ref{rem:discrete} where we highlight that Proposition~\ref{prop:quantHq} directly translates with straightforward modifications for a discrete Gaussian field.  
\end{remark}

\begin{remark}
    The previous result provides a large class of fields $B$ with long-range dependence, i.e. with covariance function $C\notin L^R$, such that the functional $Y_t:= Y(t,\ldots,t)$ exhibits Gaussian fluctuations around its mean. See e.g. Example~\ref{ex:longmemory}. We refer to~\cite{MN22} and the references therein for a deeper analysis and further examples of this phenomenon. 
\end{remark}
    
Since central limit theorems for $1$-domain functionals have been extensively explored in the literature, it is not difficult to imagine how useful and powerful the implication $(a)\Rightarrow (b)$ in Theorem \ref{thm:main} can be.
It is noteworthy that a specific instance of this implication had previously been observed in the papers \cite{RST} and \cite{PR16}; however, it was restricted to a very specific context -- rectangular increments of a fractional Brownian sheet -- and was not part of a comprehensive systematic investigation, as we undertake in this work.

\vspace{0.25cm}

Since Theorem \ref{thm:main} establishes that \eqref{func} displays Gaussian fluctuations (if and only) if at least one of its marginal functionals does, a natural question arises: what happens when none of the $Y_i(t_i)$'s exhibits Gaussian fluctuations? We investigate this problem in the classical framework of regularly varying covariance functions.

Given two functions $f$ and $g$, we write $f(x)\sim g(x)$ to indicate that \begin{equation}
    \lim_{\|x\|\to 0} f(x)g(x)^{-1}=1.
\end{equation}
Recalling that $R$ denotes the Hermite rank of $\varphi$ (see \eqref{hermitedecomp}), we consider the following conditions:
\begin{itemize}
    \item $C_i$ is {regularly varying} with parameter $-\beta_i \in(-d_i/R,0)$, that is \begin{equation}
 \label{equ:regvar} C_i(z_i)=\rho_i(\|z_i\|)=L_i(\|z_i\|)\|z_i\|^{-\beta_i}\quad \quad \beta_i\in(0,d_i/R),
\end{equation}
where $L_i:(0,\infty)\rightarrow \R$ is slowly varying, i.e. $L_i(rs)/L_i(s)\rightarrow 1$ as $s\rightarrow\infty$, $\forall r>0$.
\item $C_i$ admits an absolutely continuous spectral measure $G_i(d\lambda_i)=g_i(\lambda_i)d\lambda_i$ on $\R^{d_i}$, and for some constant $c_i>0$ we have
\begin{equation}
 \label{equ:regspec}
 g_i(\lambda_i) \sim c_iL_i(1/\|\lambda_i\|) \|\lambda_i\|^{\beta_i-d_i} \quad \quad\text{as }\|\lambda_i\|\rightarrow 0.
 \end{equation}
\end{itemize}
If \eqref{equ:regvar}-\eqref{equ:regspec} hold and $R\ge2$, then it is known that $\widetilde Y_i(t_i)$ converges in distribution to a non-Gaussian random variable (see Theorem \ref{thm:DM}). In the following Theorem \ref{thm:main 2}, we give conditions so that $Y$  has non-Gaussian fluctuations either. {We introduce the new limiting object as follows. For any $i=1,\ldots,p$ we define the $\sigma$-finite measure $\nu_i$ on $\R^{d_i}$ as \begin{equation}
    \nu_i(dx_i):=\|x_i\|^{\beta_i-d_i}dx_i, 
\end{equation}
assuming $0<\beta_i<d_i/R$. We define  $\nu$ as the product measure $\nu_1\times\ldots\times \nu_p$, which is again a $\sigma$-finite measures on $\R^{d}$. The so-called multiple $R$th Wiener-It\^o integral with respect to $\nu$, denoted by $I_{\nu,R}$ can then be constructed as in \cite{M81}, see also Section~\ref{sec:prel}. Now, we set \begin{equation}
    H^{\beta_1, \ldots, \beta_p}_{R, D_1\times \ldots \times D_p} := \frac{{\rm }I_{\nu,R}\left(f_{R,D_1\times\ldots\times D_p}\right)}{\sqrt{\var\left(I_{\nu,R}\left(f_{R,D_1\times\ldots\times D_p}\right)\right)}},
\end{equation}
where $f_{R,D_1\times\ldots\times D_p}: (\mathbb R^d)^R \to \mathbb C$ is given by
     \begin{equation}\label{equ:fNonCen}
    f_{R,D_1\times\ldots\times D_p}(\lambda_1,\dots,\lambda_R) \colon = \int_{D_1\times\ldots\times D_p} e^{i\langle x,\lambda_1+\dots+\lambda_R\rangle} dx.
\end{equation}
     \begin{remark}
     For $p=1$, we recover a Hermite random variable. For this reason, the previous object may be seen as a \textbf{$p$-domain Hermite random variable}.
     \end{remark}
     }

{
\begin{theorem}\label{thm:main 2}
Let $\f:\R \rightarrow \R$ be a measurable function satisfying $\E[\varphi^2(N)]<\infty$ for $N\sim N(0,1)$, with Hermite rank $R\geq 1$ (in particular, we have $a_R\neq0$, see \eqref{hermitedecomp}).
Let $B=(B_x)_{x\in\R^d}$ be a real-valued, continuous, centered, stationary Gaussian field with unit-variance. 
Let $C:\R^d \rightarrow \R$ be the covariance function of $B$, assume it is separable in the sense of \eqref{equ:sep}, and satisfies 
\[
C^R\ge0 \quad\quad\text{ and }\quad \quad\text{\eqref{equ:regvar}-\eqref{equ:regspec} hold $\quad\forall i\in \{1,\ldots,p\}$.}
\]
Let us consider $\widetilde Y$ given by \eqref{mainquestion}. 
Then, as $t_1,\ldots,t_p\rightarrow\infty$, we have 
\begin{equation}\label{equ:noncentral}
             \widetilde{Y}(t_1,\ldots,t_p)\overset{d}{\longrightarrow} sgn(a_R) H^{\beta_1,\ldots,\beta_p}_{R,D_1\times\ldots\times D_p},
         \end{equation}
      In particular, the limit \eqref{equ:noncentral} is not Gaussian as soon as $R\ge2$.
 \end{theorem}
 }
 
Theorem \ref{thm:main 2} represents a generalization from $1$-domain to $p$-domain of the celebrated Dobrushin-Major-Taqqu (see Theorem \ref{thm:DM}). We remark that depending on the choice of the domains, the function $f$ can be computed explicitly (see e.g. Remark~\ref{rem:domainsfourier}).
In the particular case of the rectangular increments of a fractional Brownian sheet,
we note that Theorem \ref{thm:main 2} for $p=2$ is already proved in \cite{RST} and \cite{PR16}. 

\vspace{0.25cm}

\begin{remark}\label{rem:fixed}
    Analogous results (both central and non-central) hold if we fix some of the domains, i.e. considering $t_1D_1\times\dots \times t_{p-1}D_{p-1}\times D_p$ and $Y(t_1,\dots,t_{p-1},1)$. In this setting, it is possible to prove the following.
    \begin{itemize}
        \item Under the same assumptions of Theorem \ref{thm:main}, the following two assertions are equivalent:
    \begin{enumerate}
        \item[(a)]
        $\widetilde{Y}_{i}(t_i)\overset{d}{\rightarrow} N(0,1)$ as $t_i \to \infty$ for at least one $i\in\{1,\ldots,p-1\}$;   
        \item[(b)]
        $\widetilde{Y}(t_1,\ldots,t_{p-1},1)\overset{d}{\rightarrow} N(0,1)$ as $t_1,\ldots,t_{p-1} \to \infty$.
        \end{enumerate} 
    \item Analogously to Theorem \ref{thm:main 2}, if \eqref{equ:regvar}-\eqref{equ:regspec} hold for $i=1,\dots,p-1$ and $C^R\ge0$, if $G_p$ is the spectral measure of $C_p$, then the convergence of Theorem \ref{thm:main 2} still holds replacing $\nu_p$ with $G_p$, namely
    \begin{equation}
        \widetilde{Y}(t_1,\ldots,t_{p-1},1)\overset{d}{\rightarrow} \frac{{\rm sgn(a_R)}I_{\nu_1\times\ldots\times \nu_{p-1}\times G_p,R}\left(f_{R,D_1\times\ldots\times D_p}\right)}{\sqrt{\var\left(I_{\nu_1\times\ldots\times \nu_{p-1}\times G_p,R}\left(f_{R,D_1\times\ldots\times D_p}\right)\right)}}\,, \ \text{as }t_1,\dots,t_{p-1}\rightarrow\infty .
    \end{equation}
    \end{itemize}
    The proofs of these facts follow from similar arguments to the ones of Theorem \ref{thm:main} and Theorem \ref{thm:main 2} (see Section~\ref{sec:sep}).
\end{remark}

\subsection{Non separable case}

It is easy to construct examples illustrating that the normal convergence of functionals $\widetilde{Y}_i(t_i)$ is in general not enough to determine the behavior of $\widetilde{Y}(t_1,\ldots,t_p)$ when the separability of the covariance function \eqref{equ:sep} is dropped.
See Example \ref{ex:sepmatters} for such a situation.

This is why we examine in Section \ref{sec:nonsep} what can happen when we go beyond the separable case,
by investigating two different classes: Gneiting covariance functions (Section \ref{subsec:gneiting}) and additively separable covariance functions (Section \ref{subsec:additSep}).

Although not separable, the covariance functions belonging to the Gneiting class are wedged between two separable functions.
It is therefore not surprising that a comparable phenomenon can still be proven in this context, see Theorem \ref{thm:gneitingcase}.

In contrast, the situation in the additively separable case is much more complicated. We still manage to prove a kind of ``reduction'' theorem, see Theorem \ref{thm:main 4}, with however a big difference: the marginal functionals to be considered in the additively separable case are really different from the $Y_i$'s of Theorem~\ref{thm:main}.

We refer to Section 5 for details.

\subsection{Plan of the paper} 
    The paper is organized as follows. Section \ref{sec:prel} contain some needed preliminaries. In Section \ref{sec:sep} we prove Theorem~\ref{thm:main} and Theorem~\ref{thm:main 2}.  In Section \ref{sec:ex} we provide several examples where our results apply, and we compare them with the existing literature. In Section~\ref{sec:nonsep} we go beyond the separability assumption by investigating two other frameworks. Finally, in the Appendix we prove some auxiliary results.

\section{Preliminaries}\label{sec:prel}

In this section we briefly present selected results on Malliavin-Stein method and classical results for $1$-domain functionals. 

\subsection{Elements of Malliavin-Stein method}
\label{subsec:prelmalliavin}
Theorem~\ref{thm:main} is proved using the Fourth Moment Theorem by Nualart and Peccati (see \cite{cltNuaPec}) and its quantitative version by Nourdin and Peccati (see \cite[Theorem 5.2.6]{bluebook}). For all the missing details on Malliavin calculus we refer to  \cite{nualart2006malliavin} or \cite{bluebook}.

Let us fix a probability space $(\Omega,\mathcal F, \mathbb P)$. Consider a continuous, stationary, centered Gaussian field $B=(B_{x})_{x\in \R^d}$ with unit variance and {covariance function} $C:\R^d\rightarrow\R$. We assume that $\mathcal F$ is the $\sigma$-field generated by $B$. By continuity and stationarity of $B$, the covariance function $C$ is continuous. As a result, Bochner's theorem yields the existence of a unique real, symmetric\footnote{In the sense that $G(A)=G(-A)$ for every $A\in\mathcal{B}(\R^d)$.}, finite measure $G$ on $\R^d$ endowed with the Borel $\sigma$-algebra $\B(\R^d)$, called the {spectral measure} of $B$, satisfying
\begin{equation}
    \label{Bochner}
    C(x)=\int_{\R^d}e^{i\langle x,\lambda \rangle}G(d\lambda),\quad x\in\R^d.
\end{equation}
We define the real separable Hilbert space 
\begin{equation}\label{equ:defL2G} 
\mathcal{H}:=\L^2(G)=\left\{h:\R^d\rightarrow \C : \int_{\R^d}|h(\lambda)|^2 G(d\lambda)<\infty, \ \overline{h(\lambda)}=h(-\lambda)\right\},
\end{equation}
where $|\cdot|$ denotes the complex norm, endowed with the inner product\footnote{Note that $\langle h,g\rangle_{\mathcal H}$  is real for every $h,g\in\mathcal{H}$, because $G$ is symmetric and $h,g$ are even.}
\begin{equation}
    \langle h,g\rangle_{\mathcal{H}}=\int_{\R^d}h(\lambda)\overline{g(\lambda)}G(d\lambda)=\int_{\R^d}h(\lambda)g(-\lambda)G(d\lambda).
\end{equation}
Thanks to \cite[Proposition 2.1.1]{bluebook} we may consider an isonormal Gaussian process $X$ on $\mathcal{H}$, with covariance kernel \begin{equation}
    \E\left[ X_G(h) X_G(g)\right]=\langle h,g \rangle_{\mathcal{H}}.
\end{equation}
We are going to construct an explicit isomorphism. By \eqref{Bochner}, the field $\left(X_G\left(e_x\right)\right)_{x\in\R^d}$, where $e_x:=\mathrm e^{i\langle x,\cdot \rangle}$, is stationary and Gaussian, with covariance function $C$ and spectral measure $G$. Hence, the two fields share the same distribution, that is \begin{equation}\label{equ:IGP}
    \left(X_G\left(e_x\right) \right)_{x\in\R^d} \os{d}{=} (B_x)_{x\in\R^d}.
\end{equation} 
Since we study limit theorems in distribution, we assume from now on that $B_x = X_G\left(e_x\right)$ for any $x\in\mathbb R^{d}$. For $q\ge 1$, we also define the $q$th Wiener chaos as the linear subspace of $\L^2(\Omega)$ generated by $\{ H_q(X_G(h)) : h\in \mathcal H\}$, and for every $h\in\mathcal H$ such that $\|h\|_{\mathcal H}=1$, we define the $q$th Wiener-It\^o integral as \begin{equation}\label{equ:WieItoInt}
    I_q(h^{\otimes q}) = H_q(X_G(h)),
\end{equation} 
where $h^{\otimes q}:(\R^d)^q \to \R$ is such that \begin{equation}\label{equ:tensorproduct}
    h^{\otimes q} (x_1,\ldots,x_q) = \prod_{l=1}^q h(x_l).
\end{equation} 
By density, the definition of $I_q$ can be linearly extended to every function in the space $\mathcal{H}^{\odot q} = L^2_{s}((\R^{d})^q, G^{\otimes q})$ of the {\it symmetric} functions in $\mathcal{H}^{\otimes q} = L^2((\R^{d})^q, G^{\otimes q})$.  {This fact implies the following equality, which will be crucial in the next sections. Recalling \eqref{equ:IGP} and \eqref{equ:WieItoInt}, we may write
\begin{equation}
\int_{D}H_q\left(B_x\right)dx  \overset{d}{=} \int_{D}H_q\left(X_G(e_x)\right)dx =\int_{D}I_q\left(e_x^{\otimes q}\right)dx = I_{q}(f),
\end{equation}
where $f:(\R^{d})^q\rightarrow \R$ is defined as 
\begin{equation}
    f(\lambda_1,\dots,\lambda_q) := \int_{D}\mathrm e_x^{\otimes q}( \lambda_1,\dots,\lambda_q)   dx = \int_{D}\mathrm e^{i\langle x, \lambda_1+\dots+\lambda_q \rangle}  dx.
\end{equation}
}

In the following, we write $I_{G,q}$ when we need to explicitly refer to the $q$th Wiener-It\^o integral acting on $\mathcal{H}^{\odot q}$ with respect to the measure $G$. 

For $q\in\mathbb N$, $r=1,\ldots,q-1$ and $h,g$ symmetric functions with unit norm in $\mathcal{H}$, we can define the $r$th contraction of $h^{\otimes q}$ and $g^{\otimes q}$ as the (generally non-symmetric) element of $\mathcal H^{\otimes 2q-2r}$ given by \begin{equation}
    h^{\otimes q} \otimes_r g^{\otimes q} = \scp{h,g}_{\mathcal H}^{r} \,\,h^{\otimes q-r} \otimes g^{\otimes q-r}.
\end{equation}
We then extend the definition of contraction to every pair of elements in $\mathcal H^{\odot q}$. We will denote the norm in this space by $\|\cdot\|_q$.

We are finally ready to state the celebrated Fourth Moment Theorem.

\begin{theorem}[Fourth Moment Theorem, \cite{cltNuaPec}]\label{thm:4th}
    Fix $q\ge 2$, consider $(h_t)_{t>0}\subset \mathcal H^{\odot q}$ and assume that $\E[I_q(h_t)^2]\to 1$ as $t\to\infty$. Then, the following three assertions are equivalent: \begin{itemize}
        \item $I_q(h_t)$ converges in distribution to a standard Gaussian random variable $N\sim N(0,1)$;
        \item $\E[I_q(h_t)^4] \to 3=\E[N^4]$, where $N\sim N(0,1)$;
        \item $\|h_t\otimes_r h_t\|_{2q-2r}\to 0$ as $t\to\infty$, for all $r=1,\ldots, q-1$.
    \end{itemize}
\end{theorem}

\medskip

We will also need a quantitative version of the Fourth Moment Theorem, which can be stated as follows (see \cite[Theorem 5.2.6 and (5.2.6)]{bluebook}). {Let us denote by $d_{TV}$ (resp $d_W$) the total variation (resp. Wasserstein) distance (see \cite[Appendix C]{bluebook} for a rigorous definition).}
\begin{theorem}\label{thm:4thQuant}
Fix $q\ge 2$ and consider $h\in \mathcal H^{\odot q}$. 
Then
\begin{eqnarray}
\E[I_q(h)^4]-3\E[I_q(h)^2]^2
&=&\frac3q\sum_{r=1}^{q-1} rr!^2\binom{q}{r}^4(2q-2r)!
\|h\widetilde{\otimes}_r  h\|^2_{2q-2r}\\
&=&\sum_{r=1}^{q-1} q!^2\binom{q}{r}^2\left\{\|h\otimes_r h\|^2_{2q-2r}+\binom{2q-2r}{q-r}\|h\widetilde{\otimes}_r h\|^2_{2q-2r}\right\},
\end{eqnarray}
and, with $N\sim N(0,1)$,
\begin{align}
    &d_{TV}(I_q(h),N)\le \frac{2}{\sqrt{3}}\sqrt{\E[I_q(h)^4]-3\E[I_q(h)^2]^2},\\
    &{d_{W}(I_q(h),N)\le \frac{2}{\sqrt{3\pi}}\sqrt{\E[I_q(h)^4]-3\E[I_q(h)^2]^2}.} 
\end{align}
\end{theorem}


\subsection{Classical results for \texorpdfstring{$1$}{1}-domain functionals} \label{subsec:onedomain}
In this section, we state  the two most popular results in the framework of limit theorems for {\it $1$-domain} functionals. 
The first result, the celebrated Breuer-Major theorem, was first proved in \cite{BM83} in a discrete version, and then extended to several settings. Here we state the continuous version of the result. Its proof can be found e.g. in \cite{NZ19}.
\begin{theorem}[Breuer-Major]\label{thm:BM} Let $B=(B_x)_{x\in\R^d}$ be a real-valued, continuous, centered Gaussian field on $\R^d$, assumed to be stationary and to have unit-variance. Let $C$ denotes its covariance function. Let $\f:\R \rightarrow \R$ be a measurable function satisfying $\E[\varphi^2(N)]<\infty$, $N\sim N(0,1)$, with Hermite rank $R$. Let us consider \begin{equation}
Y(t):= \int_{tD} \f(B_x) dx,
\end{equation}
where $D\subset\R^d$ is a compact set, and $t\ge 0$. If $C\in L^R(\R^d)$,
then, as $t\to\infty$, \begin{equation}
\frac{Y(t)-\E[Y(t)]}{t^{d/2}}\os{d}{\longrightarrow} N(0,\sigma^2),
\end{equation}
where $\sigma^2=\vol(D) \sum_{q=R} ^{\infty} q! a_q^2 \int_{\R^d} C(x)^q dx \ge 0$. In particular, if $\sigma^2>0$, then $\var(Y(t)) \sim \sigma^2  t^d$ and a central limit theorem holds for $(Y(t)-\E[Y(t)])/\sqrt{\var(Y(t))}$.
\end{theorem}
The idea behind Breuer-Major theorem is that if the fields is not ``too correlated" at infinity (precisely, if $\int_{\R^d}|C(x)|^R dx<\infty$), then the fluctuations of the functional $Y$ are Gaussian. 
Conversely, if this is not the case, then one can have non-Gaussian fluctuations (this does not mean that we necessarily have non-Gaussian fluctuations, see e.g. \cite{MN22}).
The following Theorem \ref{thm:main 2} provides a non-central limit theorem for functionals of Gaussian fields having a regularly varying covariance function, see \eqref{equ:regvar}, satisfying \eqref{equ:regspec}. In the discrete case, a first proof for functionals with Hermite rank $R=1,2$, was given by Taqqu in \cite{Taq75}, and then generalized to any $R$ by Dobrushin and Major in \cite{DM79}. The following sticks to the continuous case and a proof can be found e.g. in \cite{LO14}.

\begin{theorem}[Dobrushin-Major-Taqqu]\label{thm:DM}
 Let $B=(B_x)_{x\in\R^d}$ be a real-valued, continuous, centered, stationary Gaussian field with unit-variance. Let $\f:\R \rightarrow \R$ be a measurable function satisfying $\E[\varphi^2(N)]<\infty$, $N\sim N(0,1)$, with Hermite rank $R$ and $R$th coefficient $a_R\neq0$, see \eqref{hermitedecomp}. Let us consider 
 \[
 Y(t):= \int_{tD} \f(B_x) dx,
 \]
 where $D$ is a compact set with $\vol(D)>0$ and $t\ge0$.
 Let $C:\R^d \rightarrow \R$ be the covariance function of $B$ and suppose that \eqref{equ:regvar}-\eqref{equ:regspec} hold for $C$ with parameter $-\beta\in(-d/R,0)$. Then, as $t\rightarrow\infty$, we have 
\begin{equation}
             \widetilde{Y}(t)\overset{d}{\rightarrow} \frac{{\rm sgn(a_R)}I_{\nu,R}\left(f_{R,D}\right)}{\sqrt{\var\left(I_{\nu,R}\left(f_{R,D}\right)\right)}},
         \end{equation}
     where $\nu$ is the measure on $\R^{d}$ defined as $\nu(dx) :=  |x|^{\beta-d} dx$, and $f_{R,D}$ is defined as
     \begin{equation}
    f_{R,D}(\lambda_1,\dots,\lambda_R) := \int_{D} e^{i\langle x,\lambda_1+\dots+\lambda_R\rangle} dx.
\end{equation}
    In particular, the limit $I_{\nu,R}\left(f_{R,D}\right)$ is not Gaussian as soon as $R\ge2$.
 \end{theorem}

\begin{remark}\label{rem:domainsfourier}
    The previous theorem shows that the limit $I_{\nu,R}\left(f_{R,D}\right)$ depends on $R$, $d$, $\beta$ and the domain $D$. In fact, the integrand is $f_{R,D}(\lambda_1,\dots,\lambda_R)=\mathcal{F}[\mathbf{1}_D](\lambda_1+\dots+\lambda_R)$, where  $\mathcal{F}[\mathbf{1}_D]$ is the Fourier transform of the indicator function of the compact set $D$. The most common choices for $D$ are the rectangles $D=\bigtimes_{i=1}^d[0,u_i]$, $u_i\in \R_+$, with Fourier transform 
    \[
    \mathcal{F}[\mathbf{1}_D](\lambda)=\prod_{j=1}^d \int_0^{u_j}e^{i\lambda_jx_j}dx_j=\prod_{j=1}^d \frac{e^{i\lambda_ju_j}-1}{i\lambda_j},
    \]
    and domains $D=\{x\in\R^d:\|x\|\le u\}$, $u\in\R_+$, with Fourier transform
    \[
    \mathcal{F}[\mathbf{1}_D](\lambda)=c_d \,J_{d/2}(u\|\lambda\|)\left(\frac{u}{\| \lambda\|}\right)^{d/2},
    \]
    where $J_{d/2}$ is a Bessel function of the first kind and order $d/2$, and $c_d$ is a positive constant, see e.g. \cite[Section 2.2]{MN22}.
\end{remark} 

\section{Proof of Theorems~\ref{thm:main} and \ref{thm:main 2}}\label{sec:sep}

We split the proofs of Theorems~\ref{thm:main} and \ref{thm:main 2} in four subsections. In Subsection \ref{subsec:Hq}, we prove Theorem~\ref{thm:main} when $\phi =H_q$, under weaker assumptions, providing also a quantitative result. In Subsection~\ref{subsec:BMpdom}, we extend Theorem \ref{thm:BM} to $p$-domain functionals. Finally, in Subsection~\ref{subsec:main proof} and \ref{subsec:noncentralproof} we prove Theorem~\ref{thm:main} and Theorem~\ref{thm:main 2}, respectively. 

\subsection{Proof of Theorem \ref{thm:main} when \texorpdfstring{$\phi = H_q$}{}}\label{subsec:Hq}

When $\phi=H_q$, the functional \eqref{func} is given by \begin{equation}
\label{Yq}
      Y(t_1,\ldots,t_p)= Y(t_1,\ldots,t_p)[q]:=\int_{t_1D_1\times\ldots\times t_pD_p} H_q(B_x)dx.
\end{equation}
Therefore, see Section~\ref{sec:prel}, we may express it as follows \begin{equation}\label{equ:defIGP}
    \int_{t_1D_1\times\ldots\times t_pD_p} H_q(B_x)dx = I_q(f(t_1,\ldots,t_p)),
\end{equation}
where $I_q:\mathcal H^{\odot q} \to \L^2(\Omega)$ stands for the $q$th Wiener-It\^o integral \eqref{equ:WieItoInt}, and $f(t_1,\ldots,t_p)\in \mathcal H^{\odot q}$ is given by \begin{equation}
    f(t_1,\ldots,t_p) := \int_{t_1D_1\times\ldots\times t_pD_p} e_{x}^{\otimes q} dx.
\end{equation}
Let us recall the definition of the marginal functionals \eqref{marginal funct}, and, accordingly, set
\begin{equation}\label{equ:defIGP2}
   Y_i{(t_i)}[q]=I_q\bigg(\int_{t_iD_i}(e_{x_i}^{(i)})^{\otimes q}dx_i \bigg) = I_q(f_i{(t_i)}), \quad i =1,\ldots, p.  
\end{equation}

\begin{proposition}\label{prop:suffHq}
    Let $B=(B_x)_{x\in\R^d}$ be a real-valued, continuous, centered, stationary Gaussian field with unit-variance. Let $C:\R^d \rightarrow \R$ be the covariance function of $B$ and assume it is separable in the sense of \eqref{equ:sep}. Then \begin{equation}\label{equ:varHq}
        \var(Y(t_1,\ldots,t_p)[q]) = (q!)^{1-p} \prod_{i=1}^p \var (Y_i{(t_i)}[q]).
    \end{equation}
    Moreover, the following holds: if there exists $i\in\{1,\ldots,p\}$ such that
    \begin{equation}
        \widetilde{Y}_{i}(t_i)[q]\overset{d}{\rightarrow} N(0,1) \quad \quad \text{ as } t_i \to \infty,
    \end{equation}
    then
        \begin{equation}
         \widetilde{Y}(t_1,\ldots,t_p)[q]\overset{d}{\rightarrow} N(0,1) \quad \quad \text{ as } t_1,\ldots,t_p \to \infty.
        \end{equation}
\end{proposition}

\begin{proof}[Proof] Note that $\var (Y_i{(t_i)}[q])  = q! \|f_i(t_i)\|^2_q$. As a result, \eqref{equ:varHq} follows from the separability of $C$: 
\begin{align}
    \var(Y(t_1,\ldots,t_p)[q]) &= q! \int_{(t_1D_1\times \ldots \times t_pD_p)^2}C(x_1-y_1,\ldots,x_p-y_p)^q dx_1 \dots dx_p dy_1 \dots dy_p  \\
    &=  \prod_{i=1}^p {\int_{(t_i D_i)^2}C_i(x_i-y_i)^q dx_idy_i} =(q!)^{1-p} \prod_{i=1}^p \var (Y_i{(t_i)}[q]). \label{equ:varHqProof}
\end{align}
If $q=1$ everything is Gaussian and there is nothing more to show. Thus, let us assume that $q\ge2$ and let us compute the norm of the $r$th contraction of $f(t_1,\ldots,t_p)$ with itself, for $r=1,\ldots,q-1$: 
    \begin{align} \label{equ:contrHq}
    & \|f{(t_1,\dots,t_p)} \otimes_r  f{(t_1,\dots,t_p)}\|_{{2q-2r}}^2 = \\
    & = \int_{(t_1D_1\times\dots\times t_pD_p)^4} C(x-z)^rC(y-u)^rC(x-y)^{q-r}C(z-u)^{q-r}dxdydudz  \\
    & =  \prod_{i=1}^p \int_{(t_iD_i)^4} C_i(x_i-z_i)^rC_i(y_i-u_i)^rC_i(x_i-y_i)^{q-r}C_i(z_i-u_i)^{q-r}dx_idy_idu_idz_i ,
\end{align}
where $x_i$ denotes the projection of $x$ onto the $i$th block of $\mathbb R^{d}=\prod_{i=1}^p\mathbb R^{d_i}$. Let us define $\tilde f:= \frac{f}{\|f\|}$. By linearity, it immediately follows that $\widetilde Y{(t_1,\dots,t_p)}[q]=I_q(\tilde f{(t_1,\dots,t_p)})$ and $\widetilde Y_i{(t_i)}[q]=I_q(\tilde f_i{(t_i)})$. Then, combining \eqref{equ:contrHq} and \eqref{equ:varHq}, we obtain \begin{equation}\label{equ:contr}
    \|\widetilde f{(t_1,\ldots,t_p)}\otimes_r \widetilde f{(t_1,\ldots,t_p)}\|_{2q-2r}^2 =  \prod_{i=1}^p  \|\widetilde f_i{(t_i)}\otimes_r \widetilde f_i(t_i)\|_{2q-2r}^2.
\end{equation}
Since for every $f\in \mathcal H^{\otimes q}$ and $r=1,\ldots,q-1$, we have \begin{equation}\label{equ:aiuto}
    \|f\otimes_r f\|_{2q-2r} ^2 \le \|f\|_{q}^4,
\end{equation}
we obtain 
\begin{equation}
\|\widetilde f{(t_1,\dots,t_p)}\otimes_r \widetilde f{(t_1,\dots,t_p)}\|_{2q-2r}^2 \le  \min \{ \|\widetilde f_i{(t_i)}\otimes_r \widetilde f_i(t_i)\|_{2q-2r}^2 : i=1,\ldots, p\}.
\end{equation}
It remains to apply Theorem \ref{thm:4th} to conclude that, if $\widetilde Y_i(t_i)\rightarrow N(0,1)$ as $t_i\rightarrow\infty$, then $\|\widetilde f_i{(t_i)}\otimes_r \widetilde f_i(t_i)\|_{2q-2r}\to 0$ for all $r\in\{1,\ldots,q-1\}$, implying that 
$\|\widetilde f{(t_1,\dots,t_p)}\otimes_r \widetilde f{(t_1,\dots,t_p)}\|_{2q-2r}\to 0$ for all $r\in\{1,\ldots,q-1\}$, implying finally that $\widetilde Y(t_1,\ldots,t_p)\overset{d}{\rightarrow}N(0,1)$. 
\end{proof}

The following result provides the converse implication of Proposition \ref{prop:suffHq}, under additional assumptions.
\begin{proposition}\label{prop:necHq}
    Let $B=(B_x)_{x\in\R^d}$ be a real-valued, continuous, centered, stationary Gaussian field with unit-variance. Let $C:\R^d \rightarrow \R$ be the covariance function of $B$ and assume it is separable in the sense of \eqref{equ:sep}. Finally, assume either that $C^q\ge 0$ or $q=3$. Then, the following holds: if \begin{equation}
         \widetilde{Y}(t_1,\ldots,t_p)[q]\overset{d}{\rightarrow} N(0,1) \quad \quad \text{ as } t_1,\ldots,t_p \to \infty,
        \end{equation}
        then there exists at least one $i\in\{1,\ldots,p\}$ such that
        \begin{equation}
            \widetilde{Y}_{i}(t_i)[q]\overset{d}{\rightarrow} N(0,1) \quad \quad \text{ as } t_i \to \infty.
        \end{equation}
        
\end{proposition}
\begin{proof} When $q=1$, we have that $\widetilde Y$ and the $\widetilde Y_i$'s are Gaussian, 
meaning that the statement is correct but empty. 

So, let us assume from now on that $q\ge2$ and that $\widetilde{Y}(t_1,\ldots,t_p)[q]\overset{d}{\rightarrow} N(0,1)$. By the Fourth Moment Theorem~\ref{thm:4th} one has that 
$\|\widetilde f{(t_1,\ldots,t_p)}\otimes_r \widetilde f{(t_1,\ldots,t_p)}\|_{2q-2r}\to 0$
as $t_1,\ldots,t_p\rightarrow\infty$ for any $r\in\{1,\dots,q-1\}$.

When $q=2$, there is only one contraction to consider. Looking at 
\eqref{equ:contr}, we deduce that at least one of the factors in \eqref{equ:contr} must go to $0$ which, by Theorem~\ref{thm:4th}, implies that at least one among the $\widetilde Y_i(t_i)$'s must have a Gaussian limit. 

Consider now the case where $q= 3$.
By Fubini, we have for every $i=1,\ldots, p$
\begin{equation}\label{equ:symContr}    \|\widetilde f_i(t_i)\otimes_{1} \widetilde f_i(t_i)\|_{4}^2=\|\widetilde f_i(t_i)\otimes_{2}\widetilde f_i(t_i)\|_{2}^2, \end{equation}
and \eqref{equ:contr} allows again to conclude. 

Finally, let us suppose that $q\ge4$ and $C^q\ge0$, that is either $q$ even, or $q$ odd and $C\ge0$. 
Since the second contraction satisfies 
\begin{equation}
    \|\widetilde f({t_1,\ldots,t_p})\otimes_{2} \widetilde f(t_1,\ldots,t_p)\|_{2q-4}\longrightarrow 0, \quad \text{as }t_1,\ldots,t_p \rightarrow\infty,
\end{equation}
we deduce from \eqref{equ:contr} that there exists $i\in\{1,\ldots,p\}$ such that
\begin{equation}
    \label{equ:midContr}
    \|\widetilde f_i{(t_i)}\otimes_{2} \widetilde f_i(t_i)\|_{2q-4}^2\rightarrow 0 \text{ as } t_i\rightarrow\infty.
\end{equation}
Let us show that this is sufficient to conclude about the asymptotic normality of $\widetilde Y_i(t_i)[q]$ when $C^q\geq 0$. For this, let us consider a positive sequence $(a_{t_i})_{t_i>0}$ with $a_{t_i}\rightarrow\infty$ as $t_i\rightarrow\infty$ (the exact expression of $a_{t_i}$ will be made precise later) and define the subset $A_{t_i}$ of $\mathbb{R}^{4d_i}$ by 
$$A_{t_i}=\left\{(x_i,y_i,u_i,z_i)\in\mathbb{R}^{4d_i} : |C_i(x_i-z_i)C_i(y_i-u_i)|\le a_{t_i} |C_i(x_i-y_i)C_i(z_i-u_i)|\right\}.$$
We can then write, for every $ 3\le r\le q-1$:
\begin{align}
    &\|  f_i(t_i)\otimes_{ r }   f_i(t_i)\|_{2q-2r}^2=\\
    &=\int_{(t_iD_i)^4} C_i(x_i-z_i)^rC_i(y_i-u_i)^rC_i(x_i-y_i)^{q-r}C_i(z_i-u_i)^{q-r} dx_idy_idu_idz_i \\
    &\le\int_{(t_iD_i)^4\cap A_{t_i}} |C_i(x_i-z_i)^rC_i(y_i-u_i)^rC_i(x_i-y_i)^{q-r}C_i(z_i-u_i)^{q-r}|  dx_idy_idu_idz_i \\
    &\hspace{1.3cm} +\int_{(t_iD_i)^4\cap (\mathbb{R}^{4d_i}\setminus A_{t_i})} |C_i(x_i-z_i)^rC_i(y_i-u_i)^rC_i(x_i-y_i)^{q-r}C_i(z_i-u_i)^{q-r}| \\
    & \hspace{10.8cm} dx_idy_idu_idz_i \\
    & \le a_{t_i}^{r-2}\int_{(t_iD_i)^4}C_i(x_i-z_i)^2C_i(y_i-u_i)^2|C_i(x_i-y_i)|^{q-2}|C_i(z_i-u_i)|^{q-2}  dx_idy_idu_idz_i  \\
    &\hspace{1.3cm}  +a_{t_i}^{-(q-r)}{\int_{(t_i D_i)^4}|C_i(x_i-z_i)|^q |C_i(y_i-u_i)|^{q} dx_idy_idu_idz_i}\\
    & = a_{t_i}^{r-2}\int_{(t_iD_i)^4}C_i(x_i-z_i)^2C_i(y_i-u_i)^2C_i(x_i-y_i)^{q-2}C_i(z_i-u_i)^{q-2}  dx_idy_idu_idz_i  \\
    &\hspace{1.3cm}  +a_{t_i}^{-(q-r)}{\int_{(t_i D_i)^4}C_i(x_i-z_i)^q C_i(y_i-u_i)^{q} dx_idy_idu_idz_i}\\
    &=a_{t_i}^{r-2}\|  f_i(t_i)\otimes_{2}  f_i(t_i)\|_{2q-4}^2+a_{t_i}^{-(q-r)}\| f_i(t_i)\|_{q}^4.
\end{align}

Summarizing, we have, for every $r\in\{3,\dots,q-1\}$,
\begin{equation}
    \label{equ:bound}
    \| \widetilde  f_i(t_i)\otimes_{ r }  \widetilde  f_i(t_i)\|_{2q-2r }^2\le a_{t_i}^{r-2}\|  \widetilde f_i(t_i)\otimes_{2}  \widetilde f_i(t_i)\|_{2q-4}^2+a_{t_i}^{-(q-r)}.
\end{equation}
Now, let us choose $a_{t_i} =\big(\|\widetilde  f_i({t_i})\otimes_{2} \widetilde  f_i({t_i})\|_{2q-4}\big)^{\frac2{2-q}}$ and
observe that $a_{t_i}\to\infty$ as $t_i\to\infty$.
Plugging into \eqref{equ:bound} allows to obtain, for every $r\in\{3,\dots,q-1\}$,
$$
\| \widetilde  f_i({t_i})\otimes_{ r }  \widetilde  f_i({t_i})\|_{2q-2r }^2\leq 2
\big(\|\widetilde  f_i({t_i})\otimes_{2} \widetilde  f_i({t_i})\|_{2q-4}\big)^{\frac{2(q-r)}{q-2}}
\to 0\quad\mbox{as $t_i\to\infty$},
$$
But this convergence of the $r$th contraction to zero also holds for $r=1$, since the norms of the $1$-contraction and the $(q-1)$-contraction are equal.
Finally, the desired conclusion follows from the Fourth Moment Theorem~\ref{thm:4th}. \end{proof}

\begin{proposition}\label{prop:quantHq}
    Let the same notations and assumptions of Proposition~\ref{prop:suffHq} prevail. In particular $N\sim N(0,1)$.
    Then, the following estimate holds: 
    \begin{align}\label{equ:quantHqLe}
        &d_{TV}(\widetilde Y{(t_1,\ldots,t_p)}[q],N)\le c_q \, \prod_{i=1}^p \sqrt{\E\big[\widetilde Y_i[q]^4\big]-3}, \\
        &{d_{W}(\widetilde Y{(t_1,\ldots,t_p)}[q],N)\le \frac{c_q}{\sqrt{\pi}} \, \prod_{i=1}^p \sqrt{\E\big[\widetilde Y_i[q]^4\big]-3},}
    \end{align}
    where $c_q=\sqrt{\frac{4}{q}\sum_{r=1}^{q-1}rr!^2\binom{q}{r}^4(2q-2r)!}$.
\end{proposition}

\begin{proof}[Proof of Proposition~\ref{prop:quantHq}]
    Recall the result of Theorem \ref{thm:4thQuant}, namely
    \begin{equation}\label{coco}
    d_{TV}(\widetilde Y{(t_1,\ldots,t_p)}[q],N)\le \frac{2}{\sqrt{3}} \sqrt{\E\big[(\widetilde Y{(t_1,\ldots,t_p)}[q])^4\big]-3},
    \end{equation}
    and
    \begin{align}
        & \E\big[(\widetilde Y{(t_1,\ldots,t_p)}[q])^4\big]-3 \label{coco2}\\
        & = \frac3q\sum_{r=1}^{q-1} rr!^2\binom{q}{r}^4(2q-2r)!
\|\widetilde f(t_1\ldots,t_p)\widetilde{\otimes}_r \widetilde f(t_1\ldots,t_p)\|^2_{2q-2r}\\
        &= \sum_{r=1}^{q-1} q!^2\binom{q}{r}^2\bigg\{\|\widetilde f(t_1\ldots,t_p)\otimes_r \widetilde f(t_1\ldots,t_p)\|^2_{2q-2r} \\
        & \hspace{6cm}
        +\binom{2q-2r}{q-r}\|\widetilde f(t_1\ldots,t_p)\widetilde{\otimes}_r \widetilde f(t_1\ldots,t_p)\|^2_{2q-2r}\bigg\}.
    \end{align}
    We deduce 
    \begin{align}
         d_{TV}(\widetilde Y{(t_1,\ldots,t_p)}[q],N)
        &\le \frac{2}{\sqrt{3}} \sqrt{\E\big[(\widetilde Y{(t_1,\ldots,t_p)}[q])^4\big]-3}\\
        &\le c_q\,\max_{r=1,\ldots,q-1} \|\widetilde f(t_1\ldots,t_p)\otimes_r  \widetilde  f(t_1\ldots,t_p)\|^2_{2q-2r}.
    \end{align}
    But, thanks to \eqref{equ:contr} and recalling \eqref{equ:aiuto}, 
   we have 
    \begin{align}
        \max_{r=1,\ldots,q-1} \|\widetilde f(t_1\ldots,t_p)\otimes_r  \widetilde  f(t_1\ldots,t_p)\|^2_{2q-2r}
        &= \max_{r=1,\ldots,q-1} \prod_{i=1}^p \|\widetilde f_i(t_i)\otimes_r \widetilde f_i(t_i)\|^2_{2q-2r}\\
        &\le  \prod_{i=1}^p \max_{r=1,\ldots,q-1} \|\widetilde f_i(t_i)\otimes_r \widetilde f_i(t_i)\|^2_{2q-2r}.
    \end{align}
    Also, as a consequence of the second equality in \eqref{coco2} (with $\widetilde{Y}_i$ instead of $\widetilde{Y}$), we have
    $$
    \|\widetilde f_i(t_i)\otimes_r \widetilde f_i(t_i)\|^2_{2q-2r} \leq \E\big[\widetilde Y_i[q]^4\big]-3,
    $$
    and the desired conclusions now easily follows.
\end{proof}

\begin{remark}\label{rem:discrete}
    In Example~\ref{exa:fBs} below we will use the previous result to improve the bound for the central convergence of the rescaled $q$th Hermite variation of the rectangular increments of the fractional Brownian sheet obtained in \cite{RST}.
\end{remark}

{\begin{remark}
   Rates of convergence for smooth distances can be obtained by means of Malliavin-Stein method for more general functions $\varphi$. In the case where $C^R \in L^1(\mathbb R^d)$, where $R$ is the Hermite rank of $\varphi$, (equivalently $C_i^R \in L^1(\mathbb R^{d_i})$ for every $i=1,\ldots,p$), several statements can be found in \cite{quantitative}, where no additional smoothness of the function $\varphi$ is required. However, we remark that in this case all the terms of the chaotic expansion of the functional contribute to the limit. Conversely, when $C_i \notin L^R(\R^{d_i})$ for some $i=1,\ldots,p$, then a bound in Wasserstein distance could be proved. Indeed, Proposition~\ref{prop:redP} shows how to control $E[(\widetilde Y - \widetilde Y[R])^2]$, see e.g. \eqref{eq:L2redP}. However, we again have a non-trivial contribution of every term of the chaotic decomposition, leading to a sub-optimal and less meaningful rates. 
\end{remark}}

\subsection{A Breuer-Major theorem for \texorpdfstring{$p$}{p}-domain functionals}\label{subsec:BMpdom}

In this subsection, we provide an extension of the celebrated Breuer-Major theorem from the classical setting of $1$-domain functionals (see Theorem \ref{thm:BM}) to the setting of $p$-domain functionals.
\begin{theorem}\label{thm:BMpdom}
    Let $B=(B_x)_{x\in\R^d}$ be a real-valued, continuous, centered, stationary Gaussian field with unit-variance. Let $\f:\R \rightarrow \R$ be a measurable function  satisfying $\E[\varphi^2(N)]<\infty$, $N\sim N(0,1)$, with Hermite rank $R\ge 1$. Let us consider $\widetilde Y$ as in \eqref{mainquestion}. Let $C:\R^d \rightarrow \R$ be the covariance function of $B$ and assume it is separable in the sense of \eqref{equ:sep}. If $C_i\in L^R(\R^{d_i})$ for any $i\in\{1,\ldots,p\}$, then, as $t_1,\ldots,t_p\to\infty$,
        \begin{equation}
        \frac{Y(t_1,\ldots,t_p)-\E[Y(t_1,\ldots,t_p)]}{t_1^{d_1/2}\ldots t_p^{d_p/2}} \os{d}{\rightarrow} N(0,\sigma^2),
        \end{equation}
    where  \begin{equation}
            \sigma^2 = \sum_{q\ge R} a_q^2 q!  \prod_{i=1}^p \vol(D_i)\int_{\R^{d_i}} C_i^q(z_i) dz_i.
        \end{equation}
    Moreover, if $\sigma^2>0$ then 
    \begin{equation}
        \var(Y(t_1,\ldots,t_p)) \sim \sigma^2\,t_1^{d_1}\ldots t_p^{d_p}
    \end{equation}
    and we have a central limit theorem for \eqref{mainquestion}, that is
    \[
    \widetilde{Y}(t_1,\dots,t_p)\os{d}{\rightarrow} N(0,1) \quad \text{ as }t_1,\ldots,t_p\rightarrow\infty.
    \]
\end{theorem}
\begin{remark}
The previous theorem may be seen as a Breuer-Major theorem for $p$-domains. Its proof follows from similar arguments to the ones in \cite[Theorem 7.2.4]{bluebook}. However, we observe that the separable assumption allows us to recover the rate of the variance as a function of the (possibly distinct) growth rates of the domains. 
\end{remark} 
\begin{proof}[Proof of Theorem~\ref{thm:BMpdom}]
We proceed as in the proof of \cite[Theorem 7.2.4]{bluebook}, using \cite[Theorem 6.3.1]{bluebook}.
First of all, we have
\begin{align}
    \frac{Y(t_1,\dots,t_p)-\E[Y(t_1,\dots,t_p)]}{t_1^{d_1/2}\ldots t_p^{d_p/2}} &= \frac{\sum_{q=R}^\infty a_q \int_{t_1D_1 \times \dots \times t_pD_p}H_q(B_x)dx}{t_1^{d_1/2}\ldots t_p^{d_p/2}}\\&=\sum_{q=R}^\infty I_q(f_q(t_1,\dots,t_p)),
\end{align}
where 
\[
f_{q}(t_1,\dots,t_p):=\frac{a_q\int_{t_1D_1 \times \dots\times t_pD_p}e_x^{\otimes q}dx}{t_1^{d_1/2}\ldots t_p^{d_p/2}}.
\]
To conclude the proof we only need to check the conditions (a)-(d) of \cite[Theorem 6.3.1]{bluebook}.
\bigskip 

\textit{Condition (a):} We need to check that
$$
\sigma^2[q]:=\lim_{t_1,\ldots,t_p\to\infty}q!\|f_q(t_1,\dots,t_p)\|^2
$$
exists in $[0,\infty)$ for each $q\geq R$.
Since $C_i\in L^R(\R^{d_i})$ (implying $C_i\in L^q(\R^{d_i})$) for every $i\in\{1,\ldots,p\}$, the change of variable $z_i=x_i-y_i$ yields
\begin{align}
    \label{formuletta}\int_{(t_iD_i)^2}C_i^q(x_i-y_i)dx_idy_i&=t_i^{d_i}\int_{\R^{d_i}}C_i^q(z_i)\vol(D_i\cap (D_i+z_i/t_i))dz_i\\
    &\sim t_i^{d_i}\vol(D_i)\int_{\R^{d_i}}C_i^q(z_i)dz_i, \quad \quad \text{as }t_i\rightarrow \infty.
\end{align}
Combining the previous equivalent with the separability of $C$, we obtain
\begin{align}
q!\|f_q(t_1,\dots,t_p)\|^2=q!a_q^2\frac{\prod_{i=1}^p\int_{(t_iD_i)^2}C_i^q(x_i-y_i)dx_idy_i}{t_1^{d_1/2}\ldots t_p^{d_p/2}}\rightarrow \sigma^2[q],
\end{align}
as $t_1,\ldots,t_p\to \infty$, where
\[
\sigma^2[q]:=q!a_q^2
\prod_{i=1}^p\vol(D_i)\int_{\R^{d_i}}C_i^q(z_i)dz_i.
\]

\bigskip

\textit{Condition (b):} We need to check that
\[
\sum_{q=R}^\infty \sigma^2[q]<\infty.
\]
Since $\|C_i\|_\infty\le 1$ for every $i=1,\dots,p$, we have 
\[
\sum_{q=R}^\infty \sigma^2[q]\le 
\left( \sum_{q=R}^\infty q!a_q^2\right)
\prod_{i=1}^p
\left(
\vol(D_i)
\int_{\R^{d_i}}|C_i(z_i)|^Rdz\right).
\]
Thanks to
$\var(\varphi(N))=\sum_{q=R}^\infty q!a_q^2<\infty$, the claim follows.

\bigskip

\textit{Condition (c):} We need to check that for every $q\ge R$ and every $r=1,\dots,q-1$
\[
\|f_{q}(t_1,\dots,t_p)\otimes_r f_{q}(t_1,\dots,t_p)\|^2\rightarrow 0 \quad \quad \text{ as }t_1,\ldots,t_p\rightarrow\infty.
\]
Since $C$ is separable, we deduce from \eqref{equ:contrHq} that
\[
\|f_{q}(t_1,\dots,t_p)\otimes_r f_{q}(t_1,\dots,t_p)\|^2=\prod_{i=1}^p\|f_{i,q}(t_i)\otimes_r f_{i,q}(t_i)\|^2,
\]
where $f_{i,q}(t_i)$ is given by
\[
f_{i,q}(t_i):=\int_{t_iD_i}(e_{x_i}^{(i)})^{\otimes q}dx_i.
\]
Therefore, it is enough to prove that
\begin{equation}\label{toshow}
\|f_{i,q}(t_i)\otimes_r f_{i,q}(t_i)\|^2\rightarrow 0
\end{equation}
for at least one $i\in\{1,\ldots,p\}$ as $t_i\rightarrow\infty$.
But \eqref{toshow} is actually true for any $i\in\{1,\ldots,p\}$,
see indeed point (c) in the proof of \cite[Theorem 7.2.4]{bluebook} (which uses that $C_i\in L^R(\R^{d_i})$).

\bigskip

\textit{Condition (d):} We need to check that
\[
\lim_{N\rightarrow\infty}\sup_{t_1,\ldots,t_p\geq 1}\sum_{q=N+1}^\infty q!\|f_{q}(t_1,\dots,t_p)\|^2=0.
\]
By \eqref{formuletta} and since $\|C_i\|_{\infty}\le 1$ for every $i=1,\dots,p$, we have
\begin{align*}
    &\sup_{t_1,\ldots,t_p\geq 1}\sum_{q=N+1}^\infty q!\|f_{q}(t_1,\dots,t_p)\|^2\\
    =&\sup_{t_1,\ldots,t_p\geq 1}\sum_{q=N+1}^\infty q!a_q^2\frac{\prod_{i=1}^p\int_{(t_iD_i)^2}C_i^q(x_i-y_i)dx_idy_i}{t_1^{d_1/2}\ldots t_p^{d_p/2}}\\
\le& \left( \sum_{q=N+1}^\infty q!a_q^2\right)\prod_{i=1}^p \left( \vol(D_i) \int_{\R^{d_i}}|C_i(z_i)|^Rdz_i\right)
\rightarrow 0 \quad \quad \text{as }N\rightarrow\infty,
\end{align*}
where $\left( \sum_{q=N+1}^\infty q!a_q^2\right)\rightarrow 0$ being the tail of a convergent series.
\end{proof}

\subsection{Reduction to \texorpdfstring{$R$}{R}th chaos and proof of Theorem~\ref{thm:main}}\label{subsec:main proof}

The chaotic decomposition of \eqref{func} is 
\begin{equation}\label{equ:chaosDecY}
    Y(t_1,\dots,t_p)=\E[{Y}(t_1,\ldots,t_p)]+ \sum_{q\ge R} a_q Y(t_1,\ldots,t_p)[q],
\end{equation}
see \eqref{Yq}-\eqref{equ:defIGP}. In particular, this decomposition gives
\begin{equation}
    \label{equ:vardecomp} \var(Y(t_1,\dots,t_p)) = \sum_{q=R}^\infty a_q^2 \var(Y(t_1,\dots,t_p)[q]).
\end{equation}
In order to prove Theorem \ref{thm:main}, we reduce the study of the functional $\widetilde Y(t_1,\dots,t_p)$ to that of $\widetilde{Y}(t_1,\dots,t_p)[R]$, the normalization of ${Y}(t_1,\dots,t_p)[R]$, thanks to the following extension of \cite[Proposition 4]{MN22}.

\begin{proposition}\label{prop:redP}
Let $B=(B_x)_{x\in\R^d}$ be a real-valued, continuous, centered, stationary Gaussian field with unit-variance. Let $\f:\R \rightarrow \R$ be a measurable function  satisfying $\E[\varphi^2(N)]<\infty$, $N\sim N(0,1)$, with Hermite rank $R$. Let us consider $\widetilde Y$ as in \eqref{mainquestion}, and $\widetilde Y[R]$ as above. Let $C:\R^d \rightarrow \R$ be the covariance function of $B$, and assume that it is separable in the sense of \eqref{equ:sep} and that it satisfies the following two hypotheses:
\begin{enumerate}
    \item $C^R\ge0$, that is, $C_i^R\ge0$ for every $i=0,\dots,p$;
    \item for some $j\in\{1,\ldots,p\}$, we have $C_j\in \bigcup_{M=R+1}^\infty \L^M(\R^{d_j})\setminus\L^R(\R^{d_j})$.
\end{enumerate}
    Then, with ${\rm sgn}(a_R)$ denoting the sign of the $R$th Hermite coefficient in \eqref{hermitedecomp},
    \begin{equation}
         \E\left[\left({\rm sgn}(a_R)\widetilde{Y}(t_1,\ldots,t_p)[R]-\widetilde{Y}(t_1,\ldots,t_p)\right)^2\right]\rightarrow 0 \quad \quad \text{ as } t_1,\ldots,t_p \to \infty.
    \end{equation}
\end{proposition}
\begin{proof}
We divide the proof in three steps.

\bigskip

\textit{Step 1: upper and lower bounds for the variance.}
By assumption, there exist $j$ and $M\geq R+1$ such that $C_j\notin L^R(\R^{d_j})$, but $C_j\in \L^M(\R^{d_j})$. Since $C_j^R\ge0$ by assumption, by doubling conditions for non-negative definite functions (see \cite{Gorbachev}) and properties of covariograms (see \cite{Galerne}), reasoning exactly as in the first step of the proof of \cite[Proposition 9]{MN22}, there exist two positive constants $c_1>c_2>0$ such that
\[
c_2\int_{\{\|x_j\|\le t_j\}}C_j^R(z_j)dz_j\le t_j^{-d_j}\var(Y_j(t_j)[R])\le c_1\int_{\{\|x_j\|\le t_j\}}C_j^R(z_j)dz_j.
\]
In particular, this implies
\begin{equation}\label{equ:vartoinfty}
    t_j^{-d_j}\var(Y_j(t_j)[R])\rightarrow\infty \quad \quad \text{as }t_j\rightarrow\infty.
\end{equation}

\bigskip

\textit{Step 2: we prove that ${\rm Var}(Y(t_1,\dots,t_p))\sim a^2_R {\rm Var}(Y(t_1,\dots,t_p)[R])$ as $t_1,\ldots,t_p\to\infty$.}
Since $|C|\le1$ and by assumption $C_i^R\ge0$ for $i\in\{1,\dots,p\}$, we have that, for any $q>R$ and 
any $j\in \{1,\ldots,p\}$,
\begin{align} \label{equ:varsplitting}
    \frac{{\rm Var}(Y(t_1,\dots,t_p)[q])}{{\rm Var}(Y(t_1,\dots,t_p)[R])}&=\frac{q!}{R!}\prod_{i=1}^p\frac{\int_{(t_iD_i)^2}C_i^q(x_i-y_i)dx_idy_i}{\int_{(t_iD_i)^2}C_i^R(x_i-y_i)dx_idy_i}\\ 
    &\le \frac{q!}{R!} \frac{\int_{(t_jD_j)^2}C_j^q(x_j-y_j)dx_jdy_j}{\int_{(t_jD_j)^2}C_j^R(x_j-y_j)dx_jdy_j}=\frac{{\rm Var}(Y(t_j)[q])}{{\rm Var}(Y(t_j)[R])}.
\end{align}
Now, by applying Cauchy-Schwarz $n$ times, we obtain 
\begin{align*}
    \frac{\int_{(t_jD_j)^2}C_j^q(x_j-y_j)dx_jdy_j} {\int_{(t_jD_j)^2}C_j^R(x_j-y_j)dx_jdy_j} & \le\left(\frac{\int_{(t_jD_j)^2}|C_j(x_j-y_j)|^{2q-R}dx_jdy_j}{\int_{(t_jD_j)^2}C_j^R(x_j-y_j)dx_jdy_j}\right)^{1/2}\\
    &\le\left(\frac{\int_{(t_jD_j)^2}|C_j(x_j-y_j)|^{4q-3R}dx_jdy_j}{\int_{(t_jD_j)^2}C_j^R(x_j-y_j)dx_jdy_j}\right)^{1/4}\\
    &\le \dots \\
    &\le \left(\frac{\int_{(t_jD_j)^2}|C_j(x_j-y_j)|^{R+2^n(q-R)}dx_jdy_j}{\int_{(t_jD_j)^2}C_j^R(x_j-y_j)dx_jdy_j}\right)^{1/2^n}\\
    &\le \left(\vol(D_j)t_j^{d_j}\frac{\int_{\{\|x_j\|\le {\rm diam}(D_j)t_j\}}|C_j(z_j)|^{R+2^n}dz_j}{\int_{(t_jD_j)^2}C_j^R(x_j-y_j)dx_jdy_j}\right)^{1/2^n}\\
    &=  \left(\vol(D_j)t_j^{d_j}R!\frac{\int_{\{\|x_j\|\le {\rm diam}(D_j)t_j\}}|C_j(z_j)|^{R+2^n}dz_j}{\var(Y_j(t_j)[R])}\right)^{1/2^n}
\end{align*}
where the last inequality follows by a change of variable $x_j-y_j=z_j$, and the fact that $q-R\ge1$.
Since $C_j\in L^{R+2^n}(\R^{d_j})$ for $n$ sufficiently large, we deduce from \eqref{equ:vartoinfty} for every $q>R$
\[
\frac{{\rm Var}(Y(t_1,\dots,t_p)[q])}{{\rm Var}(Y(t_1,\dots,t_p)[R])}\le cq! \left(t_j^{d_j}/\var(Y_j(t_j)[R])\right)^{1/2^n},
\]
where $c>0$ depends on $R$ and $n$ but \underline{not} on $q$.
Combining this with 
\eqref{equ:vardecomp}, we get
\begin{equation}\label{kivabien}
 \left|\frac{{\rm Var}(Y(t_1,\dots,t_p))}{{\rm Var}(Y(t_1,\dots,t_p)[R])}-a^2_R\right|\leq c\,\left(\sum_{q=R+1}^\infty a_q^2q!\right)\left(t_j^{d_j}/\var(Y_j(t_j)[R])\right)^{1/2^n}.
\end{equation}
Using \eqref{equ:vartoinfty}, this implies ${\rm Var}(Y(t_1,\dots,t_p))\sim a^2_R {\rm Var}(Y(t_1,\dots,t_p)[R])$ as $t_1,\ldots,t_p\to\infty$.

\bigskip

\textit{Step 3: $\E\left[\left({\rm sgn}(a_R)\widetilde{Y}(t_1,\ldots,t_p)[R]-\widetilde{Y}(t_1,\ldots,t_p)\right)^2\right]\rightarrow 0$.}
To prove this last step, considering the decomposition
\begin{align}
	&\widetilde{Y}(t_1,\ldots,t_p)-{\rm sgn}(a_R)\widetilde{Y}(t_1,\ldots,t_p)[R]\\
	&=\frac{{\rm sgn}(a_R)({Y}(t_1,\ldots,t_p)-\E[{Y}(t_1,\ldots,t_p)]-a_RY(t_1,\dots,t_p)[R])}{a_R\sqrt{{\rm Var}(Y(t_1,\dots,t_p)[R])}} \\
 &\hspace{1.3cm} + \frac{Y(t_1,\dots,t_p)-\E[{Y}(t_1,\ldots,t_p)]}{\sqrt{{\rm Var}(Y(t_1,\dots,t_p))}}\left\{1-
	\frac{1}{|a_R|}\sqrt{\frac{{\rm Var}(Y(t_1,\dots,t_p))}{{\rm Var}(Y(t_1,\dots,t_p)[R])}}
	\right\},
\end{align}
we get that
\begin{align}\label{eq:L2redP}
	& \mathbb{E} \left[\left(\widetilde{Y}(t_1,\ldots,t_p) - {\rm sgn}(a_R)\widetilde{Y}(t_1,\ldots,t_p)[R]\right)^2\right]\\
	&\leq 2\frac{\mathbb{E}[({Y}(t_1,\ldots,t_p)-\E[{Y}(t_1,\ldots,t_p)]-a_RY(t_1,\dots,t_p)[R]))^2]}{a_R^2{\rm Var}(Y(t_1,\dots,t_p)[R]) }\\
 &\hspace{1.3cm} + 2\,\left(1-
	\frac{1}{|a_R|}\,\sqrt{\frac{{\rm Var}(Y(t_1,\dots,t_p))}{{\rm Var}(Y(t_1,\dots,t_p)[R])}}
	\right)^2\label{rrhs}.
\end{align}
By the previous step, the second addend converges to $0$.
Regarding the first addend, since by \eqref{equ:vardecomp} we have
$$\mathbb{E}\bigg[({Y}(t_1,\ldots,t_p)-\E[{Y}(t_1,\ldots,t_p)]-a_RY(t_1,\dots,t_p)[R])^2\bigg]=\sum_{q=R+1}^\infty a^2_q\,{{\rm Var}(Y({t_1,\ldots,t_p})[q])},$$ we deduce from the previous step that
\begin{align}
	\frac{\mathbb{E}[({Y}(t_1,\ldots,t_p)-\E[{Y}(t_1,\ldots,t_p)]-a_RY(t_1,\dots,t_p)[R]))^2]}{a_R^2{\rm Var}(Y(t_1,\dots,t_p)[R]) }\rightarrow 0.
\end{align}
This concludes the proof of Step 3 and Proposition \ref{prop:redP}.

\end{proof}

We are now in a position to prove Theorem~\ref{thm:main}.

\begin{proof}[Proof of Theorem~\ref{thm:main}]
Let us recall that $B=(B_x)_{x\in\R^d}$ is a stationary and continuous Gaussian random field, and that its covariance function is \begin{equation}
    C(x) = C_1(x_1) \ldots \cdot C_p(x_p), \quad x_i\in\R^{d_i}, \,i=1,\ldots,p,
\end{equation}
where $C_i:\R^{d_i}\to \R$ is a non-negative definite function such that $C_i(0)=1$ for every $i=1,\ldots,p$. 

The functionals $\widetilde Y(t_1,\ldots,t_p)$ and $\widetilde  Y_i(t_i)$ are well-defined when $C^R$ is positive. Indeed, by the decomposition \eqref{equ:vardecomp} and the formula \eqref{equ:varHqProof}, we get that their variances are strictly positive.

Let us now distinguish two cases:

$(i)$ The first one is when $C_i\in \L^R(\R^{d_i})$ for every $i$. In this case, each functional $Y_i(t_i)$ satisfies Theorem~\ref{thm:BM}, hence $\widetilde Y_i(t_i)$ has a Gaussian limit. Moreover, Theorem~\ref{thm:BMpdom} applies, too, and we can conclude the same for $\widetilde Y(t_1,\ldots,t_p)$.

$(ii)$ Secondly, let us suppose that for at least one $i$ one has $C_i \in \L^M(\R^{d_i})\setminus \L^R(\R^{d_i})$ with $M>R$. Then, by Proposition~\ref{prop:redP}, it is equivalent to study the limit in distribution of $\widetilde Y(t_1,\ldots,t_p)[R]$ only. Being in a fixed chaos and because $C_j^R\ge 0$, both Proposition~\ref{prop:suffHq} and Proposition~\ref{prop:necHq} apply. Hence, if there exist $j$ (possibly equal or distinct from $i$) such that $\widetilde Y_j(t_j)[R]$ converges in distribution to a standard Gaussian, we can deduce the same for the functional $\widetilde Y(t_1,\ldots,t_p)[R]$, and vice-versa. 
To conclude, it remains to show that $\widetilde Y_j(t_j)$ converges in distribution to a standard Gaussian random variable if and only if $\widetilde Y_j(t_j)[R]$ does. When $C_j \in \L^R(\R^{d_j})$, it is a consequence of the usual Breuer-Major theorem (Theorem~\ref{thm:BM}), since in this case both $\widetilde Y_j(t_j)$ and $\widetilde Y_j(t_j)[R]$ are converging to $N(0,1)$. When, on the contrary, we have $C_j\in L^M(\R^{d_j})\setminus L^R(\R^{d_j})$, it is a consequence of the reduction theorem proved in \cite[Proposition 4]{MN22}. 
\end{proof}

\subsection{Proof of Theorem~\ref{thm:main 2}}\label{subsec:noncentralproof}
We state two results on Wiener-It\^o integrals, which are needed for our proof of Theorem~\ref{thm:main 2}. We note that they should not be considered as new, but rather as a generalization of well-known results contained in, e.g., \cite{M81}. However, for completeness we provide their detailed proofs in the Appendix.

\begin{lemma}[Change of variable formula]\label{change}
    Let $\nu, \nu'$ be real, $\sigma$-finite measures on $\R^d$ endowed with the Borel $\sigma$-algebra $\B(\R^d)$ and let us define the real separable\footnote{Note that if $\nu$ is real, $\sigma$-finite and symmetric, then $\mathcal{H}_\nu$ is a real separable Hilbert space, and we can consider an isonormal Gaussian process on it, as well as every notion and result introduced in Section \ref{subsec:prelmalliavin} for $\L^2(G)$. } Hilbert spaces $\mathcal{H}_{\nu}:=\L^2(\nu), \mathcal{H_{\nu'}}:=\L^2(\nu')$ as in \eqref{equ:defL2G}. Let us suppose that 
    \begin{equation}
        \label{abscont}
        \nu(dx)=|a(x)|^2 \nu'(dx)
    \end{equation}
    where $a$ is a complex valued, even function (i.e. $a(-x)=\overline{a(x)}$). Then, for every $h\in \mathcal{H_{\nu}}^{\odot q}$ we have 
    \[
    I_{\nu,q}(h)\overset{\rm law}{=}I_{\nu',q}( a^{\otimes q}h),
    \]
    where $a^{\otimes q}$ is as in \eqref{equ:tensorproduct}.
\end{lemma}

 \begin{lemma}\label{change2}
     Let us denote by $I_q$ the Wiener-It\^o integral acting on $\mathcal{H}^{\odot q}$, with respect to the Lebesgue measure $dx$ on $\R^d$. Then, for every $h\in \mathcal{H}^{\odot q}=L^2_{s}((\R^{d})^q,dx)$ and $s_1,\dots,s_d>0$, we have
     \[
    I_{q}(h)\overset{\rm law}{=} s_1^{-q/2}\ldots s_d^{-q/2} I_{q}\left(\tilde h(s_1,\ldots,s_d)\right),
    \]
    where $\tilde h(s_1,\ldots,s_d)\in \mathcal H^{\odot q}$ is such that \begin{equation}
        \tilde h(s_1,\ldots,s_d) (x_1,\ldots,x_q) := h(x_{11}/s_1,\ldots, x_{1d}/s_d, \ldots, x_{q1}/s_1,\ldots, x_{qd}/s_d ).
    \end{equation}
 \end{lemma}
 
\begin{proof}[Proof of Theorem~\ref{thm:main 2}] We divide the proof in four steps. In the sequel, $c$ will denote a positive constant which may vary depending on the instance.

 \bigskip
 
 {\it Step 1: Reduction to the $R$th chaos.}
 Since  $C^R\geq 0$ and $C_i\in \bigcup_{M=R+1}^{\infty}\L^M(\R^{d_i})\setminus \L^R(\R^{d_i})$ for every $i$, by Proposition \ref{prop:redP} we are left to study the limit in distribution of $\widetilde Y(t_1,\dots,t_p)[R]$.
 
 \bigskip
 
 {\it Step 2: Variance analysis.}
     It is a standard fact (see e.g. \cite{LO13}, \cite{LO14}, \cite{L99} or \cite{M23}) that if \eqref{equ:regvar} holds for $i$, then 
     \[
     \var(Y_i(t_i)[R])\sim c\, L_i(t_i)^R t_i^{2d_i-R\beta_i}\quad \text{as }t_i\rightarrow\infty.
     \]
     We conclude that
     \[
     \var(Y(t_1,\dots,t_p)[R])=\frac{1}{(R!)^{p-1}}\prod_{i=1}^p\var(Y_i(t_i)[R])\sim c\, \prod_{i=1}^p L_i(t_i)^R t_i^{2d_i-R\beta_i}.
     \]
    
 \bigskip
 
 {\it Step 3: A suitable expression (in law) for $Y(t_1,\dots,t_p)[R]$.} Recall the definition \eqref{equ:fNonCen}. By assumption \eqref{equ:regspec}, we have $G_i(d\lambda_i)=g_i(\lambda_i)d\lambda_i$ for every $i$. Let us define $g:\R^d\rightarrow\R_+$ as
     \[
     g(x_1,\dots,x_p)=\prod_{i=1}^p g_i(x_i).
     \] 
     We get from Lemma \ref{change} that 
     \[
     Y(t_1,\dots,t_p)[R]=I_{G,R}(f_{R,t_1D_1\times\dots\times t_pD_p})\overset{\rm law}{=}I_{R}((\sqrt{g})^{\otimes R}f_{R,t_1D_1\times \dots\times t_pD_p}),
     \]
     where $I_R$ is the $R$th Wiener-It\^o integral with respect to the Lebesgue measure.
     By applying Lemma \ref{change2}, we also obtain 
     \begin{align}
         I_{R}((\sqrt{g})^{\otimes R}f_{R,t_1D_1\times \dots\times t_pD_p}) =\left(\prod_{i=1}^p t_i^{d_i-\frac{Rd_i}{2}}\right)I_{R}\left(\left(\sqrt{g(t_1,\dots,t_p)}\right)^{\otimes R}f_{R,D_1\times \dots\times D_p}\right)
     \end{align}
     where $g(t_1,\dots,t_p):\R^d\to \R_+$ is given by
     \[
     g(t_1,\dots,t_p)(x_1,\dots,x_p)=\prod_{i=1}^p g_i(x_i/t_i).
     \]
    Then, we have obtained the following expression in distribution for $\widetilde Y(t_1,\ldots,t_p)[R]$: \begin{equation}
        \widetilde Y(t_1,\dots,t_p)[R]\overset{\rm law}{=}\frac{I_{R}\left(\left(\sqrt{g(t_1,\dots,t_p)}\right)^{\otimes R}f_{R,D_1\times \dots\times D_p}\right)}{\sqrt{\var\left(I_{R}\left(\left(\sqrt{g(t_1,\dots,t_p)}\right)^{\otimes R}f_{R,D_1\times \dots\times D_p}\right)\right)}}.
    \end{equation}

     \bigskip
 
 {\it Step 4: Proving the $L^2$ convergence.} 
 The last step of the proof  consists in showing the following $L^2$ convergence: 
     \begin{multline}
\frac{I_{R}\left(\left(\sqrt{g(t_1,\dots,t_p)}\right)^{\otimes R}f_{R,D_1\times \dots\times D_p}\right)}{\sqrt{\var\left(I_{R}\left(\left(\sqrt{g(t_1,\dots,t_p)}\right)^{\otimes R}f_{R,D_1\times \dots\times D_p}\right)\right)}} \\
 \overset{L^2(\Omega)}{\longrightarrow} \frac{I_{R}((\sqrt{g'})^{\otimes R}f_{R,D_1\times \dots\times D_p})}{\sqrt{\var\left(I_{R}((\sqrt{g'})^{\otimes R}f_{R,D_1\times \dots\times D_p})\right)}} 
 \overset{\rm law}{=} \frac{I_{\nu_1\times\dots\times \nu_p,R}\left(f_{R,D_1\times \dots\times D_p}\right)}{\sqrt{\var\left(I_{\nu_1\times\dots\times \nu_p,R}\left(f_{R,D_1\times \dots\times D_p}\right)\right)}} \label{L2conv},
\end{multline}
where the equality in distribution follows from Lemma \ref{change} and where \begin{equation}
g'(x):=\prod_{i=1}^p\|x_i\|^{\beta_i-d_i}.
\end{equation} 
By the previous steps, it is enough to prove that for some positive constant $c$ we have
 \begin{multline}
       c\prod_{i=1}^p t_i^{\frac{R(\beta_i-d_i)}{2}}L_i(t_i)^{-R/2} I_{R}\left(\left(\sqrt{g(t_1,\dots,t_p)}\right)^{\otimes R}f_{R,D_1\times \dots\times D_p}\right) \\ \overset{L^2(\Omega)}{\rightarrow} I_{R}\left((\sqrt{g'})^{\otimes R}f_{R,D_1\times \dots\times D_p}\right).
 \end{multline}
It is equivalent to show that
 \begin{align}
     \label{final}
       &\left(\prod_{i=1}^p Q_{t_i}(\lambda_{1i},\dots,\lambda_{qi})\right)f_{R,D_1\times \dots\times D_p}(\lambda_1,\dots,\lambda_R)\prod_{j=1}^R\prod_{i=1}^p\|\lambda_{ji}\|^{\frac{\beta_i-d_i}{2}}\\
       &\overset{L^2((\R^d)^R)}{\longrightarrow} {f_{R,D_1\times\dots\times D_p}(\lambda_1,\dots,\lambda_R)}\prod_{j=1}^R\prod_{i=1}^p\|\lambda_{ji}\|^{\frac{\beta_i-d_i}{2}},
 \end{align}
where the $Q_{t_i}: (\R^{d_i})^R \to \mathbb C $ are defined as
 \begin{eqnarray*}
      Q_{t_i}(\lambda_{1i},\dots,\lambda_{Ri}):&=& \, \sqrt{\prod_{j=1}^R\|\lambda_{ji}/t_i\|^{d_i-\beta_i}c_i^{-1}L_i^{-1}(t_i)g_i(\lambda_{ji}/t_i)},
 \end{eqnarray*}
 and $c_i$ are some positive constant.
 In particular, we will prove \eqref{final} with the constant $c_i$ given by assumption \eqref{equ:regspec}.
 To prove \eqref{final}, we proceed by induction on $p$. If $p=1$, then  \eqref{final} reduces to
 \[
\int_{(\R^{d_1})^R}\left|Q_{t_1}(\lambda_{11},\dots,\lambda_{R1})-1\right|^2\left|f_{R,D_1}(\lambda_{11},\dots,\lambda_{R1})\right|^2 \frac{d\lambda_{11}\dots d\lambda_{R1}}{\prod_{j=1}^R\|\lambda_{j1}\|^{d_1-\beta_1}}\rightarrow 0,
 \]
and this convergence is shown in the proof of \cite[Theorem 5]{LO14}. Now, let us assume that \eqref{final} holds for $p-1$, and let us prove the result for $p$. For this, note that $xy-1=(x-1)(y-1)+(x-1)+(y-1)$, implying
\begin{multline}
    \left(\prod_{i=1}^p Q_{t_i}(\lambda_{1i},\dots,\lambda_{Ri})\right)-1
    = \left(\left(\prod_{i=1}^{p-1} Q_{t_i}(\lambda_{1i},\dots,\lambda_{Ri})\right)-1\right)\Bigg(Q_{t_p}(\lambda_{1p},\dots,\lambda_{Rp})-1\Bigg)\\
 +\left(\left(\prod_{i=1}^{p-1} Q_{t_i}(\lambda_{1i},\dots,\lambda_{Ri})\right)-1\right)+\Bigg(Q_{t_p}(\lambda_{1p},\dots,\lambda_{Rp})-1\Bigg). 
\end{multline}
 Then, by applying the triangular inequality in $L^2((\R^d)^R)$, we get that \eqref{final} holds by inductive hypothesis.
 \end{proof}

\section{Examples}\label{sec:ex}

In this Section we collect some examples.

\begin{example}[Hermite variations of a fractional Brownian sheet (fBs)]\label{exa:fBs}
    Let us consider a fBs \begin{equation}
        W^{\alpha,\beta}=(W^{\alpha,\beta})_{(x_1,x_2)\in \R_+^2 }
    \end{equation}
    with parameter $(\alpha,\beta)\in \R_+^2$ and its rectangular increments: \begin{equation}\label{equ:exaFBs1}
        R_{x_1,x_2} := W^{\alpha,\beta}_{x_1+1,x_2+1}-W^{\alpha,\beta}_{x_1,x_2+1}-W^{\alpha,\beta}_{x_1+1,x_2}+W^{\alpha,\beta}_{x_1,x_2}.
    \end{equation}
    The field $R$ defined as above is a stationary centered Gaussian field, with covariance function given by \begin{equation}
        \cov(R_{x_1,x_2},R_{y_1,y_2}) =  r_\alpha(x_1-y_1) r_\beta(x_2-y_2)
    \end{equation}
    where \begin{equation}\label{equ:exaFBs2}
        r_H(u)=\frac{1}{2}\left(|u+1|^{2H}+|u-1|^{2H}-2|u|^{2H}\right), \quad \quad u\in\R, \ H\in(0,1).
    \end{equation}
    Notice that the covariance function of $R$ is separable in the sense of \eqref{equ:sep}. Moreover, $r_{\alpha}$ (resp. $r_{\beta}$) is the covariance function of the process defined as the (one-dimensional) increments of a fractional Brownian motion (fBm) with Hurst index $\alpha$ (resp. $\beta$). We briefly recall some results about this process, see e.g. \cite{bluebook}. A fractional Brownian motion $W^H=(W^H_t)_{t\in \R}$ of Hurst index $H$, is a centered Gaussian process such that \begin{equation}
            \E[W^H_{x_1}W^H_{y_1}] = \frac{1}{2}\left(|x_1|^{2H}+|y_1|^{2H}-|x_1-y_1|^{2H}\right), \quad x_1,y_1\in\R.
        \end{equation}
    Its increment process 
    \begin{equation}
            X_u=W^H_{u+1}-W^H_u, \quad u\in \R_+
        \end{equation}
    is known as fractional Gaussian noise. It is a centered, stationary Gaussian process with covariance function as in \eqref{equ:exaFBs2}. Note that \eqref{equ:exaFBs2} is regularly varying with parameter $2H-2$ (see \cite[(7.4.3)]{bluebook}, or \cite[Example 1]{M23}), i.e. it behaves asymptotically as \begin{equation}
            r_H(u)=H(2H-1)|u|^{2H-2} + o(|u|^{2H-2}), \quad \text{as } |u|\to\infty. 
        \end{equation}
    Breuer-Major theorem (see Theorem~\ref{thm:BM}) and \cite[Theorem 7.4]{bluebook} allow us to deduce that the sequence \begin{equation}
        Y_N = \sum_{k=0}^{N-1} H_q(X_k), \quad N\ge 1,
    \end{equation}
    under a proper renormalization, converges to a Gaussian distribution if $H\le 1-1/2q$. Moreover, in \cite{NP09} and \cite{BN08}, the authors quantify the convergence in total variation distance. 
    \begin{theorem}[Theorem 1.1, 1.2 in \cite{BN08}]
        If $H\in (0,1-1/2q]$ and $V_N = \frac{Y_N}{\sqrt{\var(Y_N)}}$, then
        \begin{equation} \label{equ:exaFBsRates} 
            d_{TV}(V_N, N(0,1))\le c_{H,q} \begin{cases}
                N^{-1/2} &\text{ if } H \in (0,\frac{1}{2})\\
                N^{H-1} &\text{ if } H \in \left[\frac{1}{2},\frac{2q-3}{2q-2}\right] \\
                N^{{Hq-q+1/2}} &\text{ if } H \in \left(\frac{2q-3}{2q-2},1-\frac{1}{2q}\right) \\
                (\log N) ^{-1/2} &\text{ if } H=1-\frac{1}{2q}
        \end{cases} =: c_{H,q} \cdot  g (q,H,N).
        \end{equation}
        In the previous, $c_{H,q}$ a positive constant, that may vary, only dependent on $q$ and $H$. 
    \end{theorem}
    If $H>1-1/2q$, we have convergence towards a non Gaussian distribution (see \cite{DM79}).

In \cite{RST}, the authors show that a proper renormalization of \begin{equation}\label{equ:VNM}
        V_{N,M} = \sum_{x_1=0}^{N-1} \sum_{x_2=0}^{M-1} H_q(R_{x_1,x_2}), \quad N, M \in \mathbb N 
    \end{equation}
    converges to a Gaussian distribution as $N,M\to\infty$ when $\alpha\le 1-1/2q$ or $\beta \le 1-1/2q$\footnote{More precisely, they showed it for the Hermite variations of the field $N^\alpha M^\beta R_{i/N,j/M}^{\alpha,\beta}$. However, by self-similarity of the process $R$ (see \cite[Definition 2.3]{RST}), the two share the same law.}. Moreover, they bound the total variation distance between $V_{N,M}$ and its Gaussian limit. 
    
    \begin{theorem}[Theorem 3.1 in \cite{RST}]
    Let us denote by $c_{\alpha,\beta}$ a generic positive constant which depends on $\alpha$, $\beta$ and $q$, but which is independent of $N$ and $M$. We have \begin{enumerate}[label=(\arabic*)]
                \item If both $0<\alpha,\beta<1-1/2q$, then $\widetilde V_{N,M}$ converges in law to $N(0,1)$ with normalization $\phi(\alpha,\beta,N,M)=\sqrt{\frac{q!}{s_{\alpha}s_{\beta}}} N^{\alpha q - 1/2} M^{\beta q - 1/2}$. In addition \begin{equation}\begin{aligned}
                    d_{TV}(&\widetilde V_{N,M},N(0,1)) \\
                    &\le c_{\alpha,\beta}\sqrt{N^{-1} + N^{2\alpha-2}+N^{2\alpha q -2q +1}+M^{-1} + M^{2\beta-2}+M^{2\beta q -2q +1}}.
                    \end{aligned}
                \end{equation}
                \item If $0<\alpha<1-1/2q$ and $\beta=1-1/2q$, then $\widetilde V_{N,M}$ converges in law to $N(0,1)$ with normalization $\phi(\alpha,\beta,N,M)=\sqrt{\frac{q!}{s_{\alpha}\iota_{\beta}}} N^{\alpha q - 1/2} M^{q-1}(\log M)^{-1/2}$. In addition \begin{equation}
                    d_{TV}(\widetilde V_{N,M},N(0,1))\le c_{\alpha,\beta}\sqrt{N^{-1} + N^{2\alpha-2}+N^{2\alpha q -2q +1}+(\log M)^{-1}}.
                \end{equation}
                If $0<\beta<1-1/2q$ and $\alpha=1-1/2q$, then we get an analogous estimate as the previous one.
                \item If both $\alpha=\beta=1-1/2q$, then $\widetilde V_{N,M}$ converges in law to $N(0,1)$ with normalization $\phi(\alpha,\beta,N,M)=\sqrt{\frac{q!}{\iota_{\alpha}\iota_{\beta}}} N^{q-1}(\log N)^{-1/2} M^{q-1}(\log M)^{-1/2}$. In addition \begin{equation}
                    d_{TV}(\widetilde V_{N,M},N(0,1)) \le c_{\alpha,\beta}\sqrt{(\log N)^{-1}+(\log M)^{-1}}.
                \end{equation} 
                \item If $\alpha<1-1/2q$ and $\beta>1-1/2q$, then $\widetilde V_{N,M}$ converges in law to $N(0,1)$ with normalization $\phi(\alpha,\beta,N,M)=\sqrt{\frac{q!}{s_{\alpha}\kappa_{\beta}}} N^{\alpha q - 1/2} M^{ q - 1}$. In addition\footnote{The exponent in red is the correction of a typo in \cite[Theorem 3.1]{RST}.} \begin{equation}
                    d_{TV}(\widetilde V_{N,M},N(0,1)) 
                    \le c_{\alpha,\beta}\sqrt{N^{-1} + N^{2\alpha-2}+N^{2\alpha q -2q +1}+M^{\color{red}{-(2\beta q -2q +1)}}}.
                    \end{equation}
                \item If $\alpha=1-1/2q$ and $\beta>1-1/2q$, then $\widetilde V_{N,M}$ converges in law to $N(0,1)$ with normalization $\phi(\alpha,\beta,N,M)=\sqrt{\frac{q!}{\iota_{\alpha}\kappa_{\beta}}} N^{ q - 1} (\log N)^{-1/2} M^{ q - 1}$. In addition \begin{equation}
                    d_{TV}(\widetilde V_{N,M},N(0,1)) 
                    \le c_{\alpha,\beta}\sqrt{(\log N)^{-1}+M^{-(2\beta q -2q +1)}}
                    \end{equation} 
            \end{enumerate}            
        \end{theorem}
        
        Our Theorem~\ref{thm:main} translates in the discrete setting, too, and returns the same qualitative phenomenon described above. Moreover, by using the bounds obtained in the proof of \cite[Theorem 4.1]{NP09}, Proposition~\ref{prop:quantHq} allows us to improve the previous rates, {by replacing a sum of powers or logarithmic functions of $N$ and $M$ with their product. This results in a smaller bound and, more interestingly, clarifies that the growth rate of every domain (the sums in \eqref{equ:VNM}) contributes multiplicatively to the speed of convergence rather than additively. }
        \begin{corollary}\label{co:exaFBsBounds} Recall the definition of $g$ in \eqref{equ:exaFBsRates}. 
        \begin{enumerate}[label=(\arabic*)]
            \item If both $0<\alpha,\beta\leq 1-1/2q$, then \begin{equation}
                d_{TV}(\widetilde V_{N,M},N(0,1))  \le c_{\alpha,\beta,q} \ g(q,\alpha,N) \cdot g(q,\beta,M).
            \end{equation}
            \item If $\alpha\leq 1-1/2q$ and $\beta>1-1/2q$, then \begin{equation}
                d_{TV}(\widetilde V_{N,M},N(0,1))  \le c_{\alpha,\beta,q} \ g(q,\alpha,N),
            \end{equation}
        \end{enumerate}
        where $c_{\alpha,\beta,q}$ are constants depending on $\alpha,\beta,q$ that may differ from the ones in \cite[Theorem 3.1]{RST}, and the normalization terms are as before.
        \end{corollary}
       
\end{example}

\begin{example}[Tensor product of regularly varying covariance functions] Let us define a centered Gaussian field $B=(B_x)_{x\in\R^d}$ with separable covariance function as in \eqref{equ:sep}, with $p\ge 2$ and  $C_i:\R^{d_i}\to \R$ defined as \begin{equation}
    C_i(x_i):= \frac{1}{(1+\|x_i\|^2)^{\beta_i/2}}, \quad \beta_i>0,\quad i=1,\ldots,p.
\end{equation} 
Let us consider the quadratic variation of this field, that is, the functional $Y(t_1,\ldots,t_p)[2]$ as in \eqref{equ:defIGP} as well as its rescaled version $\widetilde Y$. The parameters $\beta_i$ determine the behavior of the functional. Indeed, combining Theorem~\ref{thm:BM} and Proposition~\ref{prop:suffHq}, it is enough to have $\beta_i>d_i/2$ for at least one $i$ 
for $\widetilde Y(t_1,\dots,t_p)[2]$ to have a Gaussian limit. Moreover, by the application of Proposition~\ref{prop:quantHq}, we may also deduce an upper bound for the rate of convergence in total variation. For a fixed $i$ satisfying $\beta_i> d_i/2$, an upper bound for the rate of convergence in total variation distance of $\widetilde Y_i(t_i)[2]$ (recall \eqref{stand marginal funct}) towards a Gaussian distribution is given by the rate of convergence to $0$ of the norm of the contraction:  
\begin{equation}
    \|\widetilde f_i(t_i)\otimes_1 \widetilde  f_i(t_i)\|^2_2 \lesssim \frac{\left(\int_{{\|x_i\|\le t_i}}C_i(x_i)dx_i\right)^2}{t_i^{d_i}} \asymp \begin{cases}
                t_i^{-d_i} &\text{ if } \beta_i>d_i\\
                \log(t_i)^2t_i^{-d_i} &\text{ if } \beta_i=d_i \\
                t_i^{d_i-2\beta_i} &\text{ if } \beta_i\in(d_i/2,d_i) 
        \end{cases} =: g(\beta_i,t_i)^2,
\end{equation}
where we denote by $\lesssim$ the inequality up to a positive constant. From the above equivalence and  \eqref{equ:quantHqLe}, we obtain that, if $\beta_i>d_i/2$, \begin{equation}
	d_{TV}(\widetilde Y_i(t_i)[2],N(0,1)) \lesssim  g(\beta_i,t_i) .
\end{equation} 
Denoting $J = \{i\in \{1,\ldots,p\}: \beta_i> d_i/2\}$, we can deduce that \begin{equation}
	d_{TV}(\widetilde Y(t_1,\ldots,t_p)[2],N(0,1)) \lesssim \prod_{i\in J} g(\beta_i,t_i) .
\end{equation}
\end{example}

\begin{example}[Gaussian fluctuations in a long-range dependence setting]\label{ex:longmemory}
Let us define a centered Gaussian field $B=(B_x)_{\R^2}$ with separable covariance function
 \begin{equation}
    C(x_1,x_2) = \frac{1}{(1+|x_1|^2)^{1/4}} \cdot \frac{1}{(1+|x_2|^2)^{3/2}}.
\end{equation}
Let us consider \begin{equation}
    Y(t_1,t_2) = \int_{t_1D_1\times t_2D_2} H_2(B_x) dx, \quad \quad t_1,t_2>0,
\end{equation}
where $H_2$ is the second Hermite polynomial and $D_i\subset R$ are compact sets. Then, $R=2$. We have that $C \notin L^2(\R^2)$, hence we are in the long-range dependence case. However, the marginal functional \begin{equation}
    \widetilde Y_2(t_2) = \frac{\int_{t_2D_2} H_2(B^{(2)}_x) dx}{\sqrt{\var\left(\int_{t_2D_2} H_2(B^{(2)}_x) dx\right)}}, 
\end{equation}
see also \eqref{marginal funct}, where $(B^{(2)}_x)_{\R}$ is  centered Gaussian field with covariance function $C_2(x_2)=\frac{1}{(1+|x_2|^2)^{3/2}}$, exhibits Gaussian fluctuations as $t_2\to \infty$ by Theorem~\ref{thm:BM}. Hence, by Theorem~\ref{thm:main} also $\widetilde Y(t_1,t_2)$ does.
\end{example}

\section{Going beyond the separability assumption}\label{sec:nonsep}
While the separability assumption holds true in numerous applications (see, e.g., \cite[Chapter 5]{Christakos1992RandomFM} for examples in hydrology and fluid dynamics, or think of the many frameworks where the fractional Brownian sheet arises, see, e.g., \cite{OZ,SW17}), there also many instances where it does not.

In this section, we aim to examine what can happen when we go beyond the separable case.
First, let us give a counterexample that shows that we cannot always expect a result as simple as Theorem~\ref{thm:main} 
in the non separable context. 

\begin{counterexample}[Separability matters!]\label{ex:sepmatters} Let us consider $p=2$, $d_1=d_2=1$, and thus $d=2$. Fix $R\ge 2$ and consider $\alpha \in (1/R,2/R)$. Let $B=(B_x)_{\R^d}$ be a centered stationary Gaussian field with unit-variance and covariance function $C:\R^d\to \R$ as \begin{equation}
    C(x) = \frac{L(\|x\|)}{\|x\|^\alpha }, \quad x\in \R^d\setminus\{0\},
\end{equation}
with $L$ a slowly varying function, such that $C$ is continuous in $0$. Assuming that $C$ satisfies \eqref{equ:regspec}, since $\alpha<\frac2R$,  Theorem~\ref{thm:DM} applies to the functional \begin{equation}
    Y(t,t)[R] = \int_{t(D_1\times D_2)} H_R(B_x)dx
\end{equation} 
and we obtain that $\widetilde Y(t,t)[R]$ is asymptotically not Gaussian. {Considering the two induced covariance functions $C_{1}(x_{1}):=C(x_{1},0)$ and $C_{2}(x_{2}):=C(0, x_{2})$, we may construct two stationary centered Gaussian random fields $(B^{(i)}_{x_{i}})_{x_{i}\in\mathbb R^{d_{i}}}$, $i=1,2$, having $C_{i}$ as covariance function and the corresponding additive functional $\widetilde Y_i(t)[R]$ on the domain $D_{i}$, as also described in \eqref{stand marginal funct}.} However, both $\widetilde Y_1(t)[R] $ and $\widetilde Y_2(t)[R] $ have Gaussian fluctuations as $t\to\infty$. Since \begin{equation}
    C_i(x_i) = \frac{L(|x_i|)}{|x_i|^\alpha } , \quad i=1,2;
\end{equation} 
with $\alpha > 1/R$, they both satisfy Theorem~\ref{thm:BM}. This counterexample shows how the convergence of both functionals $Y_i(t_i)$ is in general not enough to determine the behavior of $Y(t_1,\ldots,t_p)$, when the separability of the covariance function \eqref{equ:sep} is not satisfied, {even in the particular case $t_{1}=t_{2}=t$}.
\end{counterexample}


In what follows, we drop the separability assumption by examining
two different classes: Gneiting covariance functions (Section \ref{subsec:gneiting}) and additively separable covariance functions (Section \ref{subsec:additSep}). We will explore whether, akin to the separable case, it is possible to simplify the asymptotic study of 
$Y(t_1,\dots,t_p)$ given by \eqref{func} by reducing it to that of simpler functionals $Y_i(t_i)$.

\subsection{Gneiting covariance functions}\label{subsec:gneiting}

The class of Gneiting covariance functions was first introduced by Gneiting in \cite{G02}. These functions are very popular in many applications, including geostatistics, environmental science, climatology and meteorology.  A \textbf{Gneiting covariance function} $C:\R^d\rightarrow \R$ has the form 
\begin{equation}\label{equ:gneiting}
    C(x_1,x_2)=\psi\left(\|x_2\|^2\right)^{-d_1/2}\varphi\left(\frac{\|x_1\|^2}{\psi\left(\|x_2\|^2\right)}\right),\quad \quad x_1\in\R^{d_1}, x_2\in\R^{d_2},
\end{equation}
where $\varphi:[0,\infty)\rightarrow (0,\infty)$ is a \textbf{completely monotone function}, that is
\[
(-1)^n\varphi^{(n)}(t)\ge0 \quad \quad \forall n\in\mathbb{N}, \ t\ge 0,
\]
and $\psi:[0,\infty)\rightarrow \R_+$ has
completely monotone derivative. The fact that \eqref{equ:gneiting} defines a non-negative definite function on $\R^{d_1+d_2}$ if $d_2=1$ was proved in \cite{G02}, and can be extended to $d_2\ge1$ using similar arguments. Assuming that $B$ is a centered Gaussian random field with unit-variance and covariance function as in~\eqref{equ:gneiting}, we may infer that $C(0,0)=\psi(0)^{-d_1/2}\varphi(0)=1$, and suppose, without loss of generality, that $\varphi(0)=\psi(0)=1$.
As a consequence, the Gaussian fields  $(B^{(1)}_{x_1})_{x_1\in\R^{d_1}}:=(B_{x_1,0})_{x_1\in\R^{d_1}}$ and $(B^{(2)}_{x_2})_{x_2\in\R^{d_2}}:=(B_{0,x_2})_{x_2\in\R^{d_2}}$ have the following covariance functions \begin{align}
    &\label{c1gneiting}C_1(x_1) :=  C(x_1,0) = \varphi\left(\|x_1\|^2\right); \\
    &\label{c2gneiting}C_2(x_2) := C(0,x_2) = \psi\left(\|x_2\|^2\right)^{-d_1/2}.
\end{align}
For this reason, we rewrite \eqref{equ:gneiting} as 
\begin{equation}
    \label{equ:gneiting alternative}C(x_1,x_2)=C_2(x_2)C_1\left(x_1C_2(x_2)^{2/d_1}\right), \quad \quad x_1\in\R^{d_1}, x_2\in\R^{d_2}.
\end{equation}
From the expression \eqref{equ:gneiting alternative}, since $C_1(x_1)$, $C_2(x_2)$ are positive and non-increasing in the norms $\|x_1\|$, $\|x_2\|$, we deduce that 
\begin{equation}
    \label{equ:gneitingproperties}
    C_2(x_2)C_1(x_1)\le C(x_1,x_2)\le C_2(x_2)C_1\left(x_1\,\underline{C_2}(D_2)^{2/d_1}\right),
\end{equation}
where $D_2$ denotes the domain of the variable $x_2$, and
\begin{equation}
    \underline{C_2}(D_2)=\inf_{x_2\in (D_2-D_2)}C_2(x_2)=\psi\left({\rm diam}(D_2)^2\right)^{-d_1/2}.
\end{equation}
Therefore, $C(x_1,x_2)$ is wedged between two separable covariance functions, that is, the product $C_1(x_1)C_2(x_2)$ and the function $C_1(x_1\, \underline{C_2}(D_2)^{2/d_1})C_2(x_2)$. This suggests that the Gneiting case may  be studied combining the bounds \eqref{equ:gneitingproperties} and Theorem \ref{thm:main}. We will partially formalize this intuition in Theorem \ref{thm:gneitingcase}, and we conjecture that the latter could be extended to more general functionals, growing domains and classes of non-separable covariance functions satisfying properties analogous to \eqref{equ:gneitingproperties}. These extensions are left for future research.

Analogously to the version of Theorem~\ref{thm:main} explained in Remark~\ref{rem:fixed}, Theorem \ref{thm:gneitingcase} allows to reduce the study of 
\[
Y(t_1,t_2)[q]=\int_{t_1D_1\times t_2D_2}H_q(B_{x})dx, \quad \quad \text{ as }t_i\rightarrow\infty,
\]
for only one index $i\in\{1,2\}$ (i.e. only the $i$-th domain is growing), to that of the respective marginal functional
\[
Y_i(t_i)[q]=\int_{t_iD_i}H_q(B^{(i)}_{x_i})dx_i, \quad \quad\text{as }t_i\rightarrow\infty,
\]
where $D_i\subseteq \R^{d_i}$ are compact sets with $\vol(D_i)>0$ and $(B^{(i)}_{x_i})_{x_i\in \R^{d_i}}$ are Gaussian fields with covariance functions $C_i$, as defined in \eqref{c1gneiting}.

\begin{theorem}\label{thm:gneitingcase}
    Let $B=(B_x)_{x\in\R^d}$ be a real-valued, continuous, centered, stationary Gaussian field with unit-variance. Let $C:\R^d \rightarrow \R$ be the covariance function of $B$, and assume that $C$ is a Gneiting covariance function, see \eqref{equ:gneiting}-\eqref{equ:gneiting alternative}. Let $Y(t_1,t_2)[q]$, $Y_i(t_i)[q]$ be defined as above, and $\widetilde Y$, $\widetilde Y_i$ their normalized versions (see e.g. \eqref{mainquestion}). Fix $j\in\{1,2\}$. Then, the following holds:
    $$\widetilde{Y}_j(t_j)[q]\overset{d}{\rightarrow} N(0,1),\quad \quad \text{ as }t_j\rightarrow\infty$$ 
    implies
    $$\widetilde{Y}(t_1,t_2)[q]\overset{d}{\rightarrow} N(0,1), \quad \quad \text{ as }t_j\rightarrow\infty$$
    where $t_k$ is fixed, for $k\neq j$.
\end{theorem}

\begin{proof}[Proof of Theorem \ref{thm:gneitingcase}] The case $q=1$ is trivial because everything is Gaussian, so let us focus on $q\ge2$. We will show the convergence by means of the Fourth Moment Theorem \ref{thm:4th}. First, recall that
\[
\var\left(Y(t_1,t_2)[q]\right)=q!\int_{(t_1D_1\times t_2D_2)^2}C^q(x-y)dx dy.
\]
Thus, by \eqref{equ:gneitingproperties} and \eqref{equ:varHq}, we have
\begin{equation}\label{equ:var lowerbound}
    \var\left(Y_1(t_1)[q]\right)\var\left(Y_2(t_2)[q]\right)\le q!\var\left(Y(t_1,t_2)[q]\right).
\end{equation}
For the sake of brevity, let use denote $h(x_2):=C_2(x_2)^{2/d_1}$. We have
\begin{align}
    & \|f(t_1,t_2)\otimes_r f(t_1,t_2) \|^2   \\
    & \qquad = \int_{(t_1D_1 \times t_2D_2)^4}  C_2(x_2-y_2)^rC_2(z_2-u_2)^rC_2(x_2-z_2)^{q-r}C_2(y_2-u_2)^{q-r} \\ 
    & \qquad\hspace{0.5cm}  \times C_1((x_1-y_1) h(x_2-y_2))^r C_1((z_1-u_1)h(z_2-u_2))^r \\ 
    & \qquad\hspace{0.5cm}  \times C_1((x_1-z_1)h(x_2-z_2))^{q-r}C_1((y_1-u_1)h(y_2-u_2))^{q-r}dx \ dy \ dz \ du.
\end{align}
If we fix $t_1$ and let $t_2\rightarrow\infty$, we get from $0\leq C_1\le1$ that 
    \begin{equation}
         \|f(t_1,t_2)\otimes_r f(t_1,t_2) \|^2 \le  \|f_2(t_2)\otimes_r f_2(t_2) \|^2 \vol(t_1D_1)^4.
\end{equation}
    If we fix $t_2$ and let $t_1\rightarrow\infty$, setting $A:=\underline{C_2}(t_2D_2)^{2/d_1}<\infty$, by \eqref{equ:gneitingproperties}  we obtain
    \begin{align}
    & \|f(t_1,t_2)\otimes_r  f(t_1,t_2) \|^2  \\
    & \quad \le \int_{(t_1D_1 \times t_2D_2)^4} dx \ dy \ dz \ du \\
    &\quad \hspace{0.5cm}  \times C_2(x_2-y_2)^rC_2(z_2-u_2)^rC_2(x_2-z_2)^{q-r}C_2(y_2-u_2)^{q-r}\\
         &\quad \hspace{0.5cm}  \times C_1((x_1-y_1)A)^r C_1((z_1-u_1)A)^r \ C_1((x_1-z_1)A)^{q-r}C_1((y_1-u_1)A)^{q-r} \\
       &\quad = A^{-4d_1}\int_{(At_1D_1 \times t_2D_2)^4} C_2(x_2-y_2)^rC_2(z_2-u_2)^rC_2(x_2-z_2)^{q-r}C_2(y_2-u_2)^{q-r}  \\ &\quad \hspace{0.5cm}  \times C_1(x_1-y_1)^r C_1(z_1-u_1)^r C_1(x_1-z_1)^{q-r}C_1(y_1-u_1)^{q-r}  dx \ dy \ dz \ du   \\
       &\quad \le A^{-4d_1}\, \vol(t_2 D_2)^4  \|f_1(t_1)\otimes_r f_1(t_1) \|^2,
\end{align}
    {where the first inequality is due to \eqref{equ:gneitingproperties} and \eqref{c1gneiting}, which imply $C_{1}(x_{1}h(x_{2}))\le C(x_{1}A)$ for all $x_{i}\in\mathbb R^{d_{i}}$, $i=1,2$; the equality in the middle follows by change of variable $x\leftarrow Ax$; and the last inequality follows since $C_{2}\le1$ and recalling the definition of $\|f_1(t_1)\otimes_r f_1(t_1) \|^2$.}
    Applying the Fourth-Moment Theorem \ref{thm:4th} leads to the desired conclusion.
\end{proof}

\subsection{Additively separable covariance functions}\label{subsec:additSep}

In this subsection we consider \textbf{additively separable covariance functions} of the form
\begin{equation} \label{equ:addsep}
    C(x_1,x_2) = K_1(x_1)+K_2(x_2), \quad \quad x_1\in\R^{d_1}, \ x_2\in\R^{d_2},
\end{equation}
where $K_1$ and $K_2$ are the covariance functions of two stationary, continuous, centered Gaussian fields $(B^{(1)})_{x_1\in\R^{d_1}}$ and $(B^{(2)})_{x_2\in\R^{d_2}}$, with $K_1(0),K_2(0)>0$. From the point of view of applications, the study of Gaussian fields with additively separable covariance function \eqref{equ:addsep} is motivated by the fact that they can model the sum of two independent Gaussian fields.  
Unlike the Gneiting class considered in the previous subsection, which is comparable to the separable case (see \eqref{equ:gneitingproperties}), here reduction theorems have to be developed in a different way. In this case, we define the marginal functionals as
\begin{equation}
\label{marginal additive funct}
A_i(t_i)[q]=\int_{t_iD_i}H_q\left(\frac{B^{(i)}_{x_i}}{K_i(0)}\right)dx_i. 
\end{equation} 
We introduce new quantities to describe the interplay between the growth rates of the volumes and the variances, that are the quotients
\begin{equation}
\label{quotients}
    \gamma^i_{t_i}[q] := \frac{\int_{(t_iD_i)^2}K_i^q(x_i-y_i)dx_idy_i}{\vol(t_iD_i)^2}.
    \end{equation}
\begin{theorem}\label{thm:main 4}
Let $B=(B_x)_{x\in\R^d}$ be a real-valued, continuous, centered, stationary Gaussian field with unit-variance. Let $C:\R^d \rightarrow \R$ be the covariance function of $B$, and assume it is additively separable in the sense of \eqref{equ:addsep}, with $K_1,K_2\ge0$. Let $\widetilde Y$ be as in \eqref{mainquestion}. For $i=1,2$, let $A_i(t_i)[q]$ be as in \eqref{marginal additive funct}, $\widetilde A_i$ its normalized versions (see e.g. \eqref{mainquestion}), and recall the definition \eqref{quotients} of $\gamma^i_{t_i}[q]$. Finally, assume that 
 \[
 \frac{\gamma^1_{t_1}[q]}{\gamma^2_{t_2}[q]}\rightarrow0 \quad \text{as $t_1,t_2\rightarrow\infty$}.
 \]
 Then, the following holds:
 \[
 \widetilde{A}_2(t_2)[q]\overset{d}{\rightarrow} N(0,1), \quad\text{ as }t_2\rightarrow\infty ,
 \]
 if and only if
 \[
 \widetilde{Y}_1(t_1,t_2)[q]\overset{d}{\rightarrow} N(0,1), \quad \text{ as }t_1,t_2\rightarrow\infty.
 \]
\end{theorem}
\begin{remark}
    Note that, by symmetry, the roles of $K_1$ and $K_2$ can be exchanged in the statement of Theorem \ref{thm:main 4}, obtaining an equivalence between $\widetilde{A}_1(t_1)[q]\overset{d}{\rightarrow} N(0,1)$ and $\widetilde{Y}(t_1,t_2)[q]\overset{d}{\rightarrow} N(0,1)$ as $t_1,t_2\rightarrow\infty$, if $\frac{\gamma^2_{t_2}[q]}{\gamma^1_{t_1}[q]}\rightarrow0$. 
\end{remark}

\begin{remark}\label{rem:Quotients}
To ease the computation of the quotients $\gamma^i_{t_i}[q]$, $i=1,2$, as noted in Equation \eqref{equ:remQuotients}, assuming that $K_i$ is both non-negative and non-negative definite, one can observe that \begin{equation}
    \gamma_{t_i}^i[q]\asymp\frac{\int_{ \{ \|x\|\le t_i \}} K_i^q(x_i)dx_i}{{t_i}^{d_i}} .
\end{equation}
\end{remark}

Just like Theorem \ref{thm:main} and Theorem \ref{thm:gneitingcase}, Theorem \ref{thm:main 4} should be interpreted as a reduction theorem, since it allows to reduce the asymptotic problem of a $2$-domain functional ${Y}(t_1,t_2)[q]$ to that of a $1$-domain functional.
The difference here is that the marginal functionals to be considered in the additively separable case are not the same considered for the separable and Gneiting classes. Moreover, unlike the latter cases, here the growth rates of the integration domains come into play by means of the quotients $\gamma_{t_i}^i[q]$ in \eqref{quotients} (see also Example~\ref{exa:slogt}). In order to prove Theorem \ref{thm:main 4}, we state and prove the following lemma.
\begin{lemma}\label{lem:var_sum_2} Under the notations and assumptions of Theorem~\ref{thm:main 4}, we have that
    \begin{multline}\label{equ:thm_red_sum_2}
        \var(Y({t_1,t_2})[q]) = \vol(t_1D_1)^2  \var(A_2({{t_2}})[q])+\var(A_1({t_1})[q])\vol({t_2}D_2)^2\\
        +\sum_{k=1}^{q-1} \binom{q}{k}^2 \var(A_1(t_1)[k])\var(A_2({t_2})[q-k]).
    \end{multline}
    In addition, assuming $q\ge2$, we have that
    \begin{equation}
     \frac{\var(Y(t_1,t_2)[q])}{\vol(t_1D_1)^2 \var(A_2(t_2)[q])}\longrightarrow 1\quad \quad \text{ as }t_1,t_2\rightarrow\infty.
    \end{equation}
\end{lemma}
\begin{proof} Recall $q\ge2$. Applying Newton's binomial formula, we may write
        \begin{align}
        & \var(Y(t_1,{t_2})[q])  = q! \int_{({t_1}D_1\times {t_2}D_2)^2} C(x-y)^q dxdy \\
        &\quad = q! \int_{(t_1D_1\times t_2D_2)^2} (K_1(x_1-y_1)+K_2(x_2-y_2))^q dx_1dx_2dy_1dy_2 \\
        &\quad = \sum_{k=0}^q q! \binom{q}{k}\int_{(t_1D_1)^2} K_1(x_1-y_1)^kdx_1dy_1\cdot \int_{( t_2D_2)^2}K_2(x_2-y_2)^{q-k} dx_2dy_2,
        \end{align}
        which is \eqref{equ:thm_red_sum_2}.
        Now assume that $\gamma^1_{t_1}[q]/\gamma^2_{t_2}[q]\rightarrow 0$. Then, as $t_1, t_2\rightarrow\infty$  \begin{equation*}
        \frac{\var(Y(t_1,t_2)[q])}{\vol(t_1D_1)^2 \var(A_2(t_2)[q])} \longrightarrow 1.
    \end{equation*}
    Indeed, \begin{equation*}
         \frac{\var(A_1(t_1)[q])\vol(t_2D_2)^2}{\vol(t_1D_1)^2  \var(A_2(t_2)[q])}=\frac{\gamma^1_{t_1}[q]}{\gamma^2_{t_2}[q]}\longrightarrow 0
    \end{equation*} 
    by assumption.
    By Jensen inequality in $t_i$ and recalling that $K_i\ge0$, we have that 
    \begin{equation}
        \var(A_i(t_i)[k]) =  k!\int_{(t_iD_i)^2}K_i^k(x_i-y_i)dx_idy_i
         \lesssim  \vol(t_iD_i)^{2(q-k)/q}\var(A_i({t_i})[q])^{k/q}.
    \end{equation}
    Therefore, the proof is concluded by observing that, for $k>0$, we have
    \begin{equation}
    \frac{\var(A_1(t_1)[k])\var(A_2(t_2)[q-k])}{\vol(t_1D_1)^2  \var(A_2(t_2)[q])} \lesssim \frac{\var(A_1(t_1)[q])^{k/q}\vol(t_2D_2)^{2k/q}}{\vol(t_1D_1)^{2k/q}\var(A_2(t_2)[q])^{k/q}} =\frac{\gamma^1_{t_1}[q]^{k/q}}{\gamma^2_{t_2}[q]^{k/q}}\longrightarrow 0.
    \end{equation}
    \end{proof}

\begin{proof}[Proof of Theorem \ref{thm:main 4}]
By assumption, $\gamma^1_{t_1}[q]/\gamma_{t_2}^2[q]\rightarrow 0$. By Lemma \ref{lem:var_sum_2}, as $t_1,t_2\rightarrow\infty$ we have
\[
\var(Y(t_1,t_2)[q])\sim \vol(t_1D_1)^2\var(A_2(t_2)[q]).
\]
Regarding contractions, by using the Newton binomial formula, we have 
\begin{align}
&\|f(t_1,t_2)\otimes_r f(t_1,t_2) \|^2\\
&\quad =\int_{(t_1D_1 \times t_2D_2)^4}  C(x-y)^rC_2(z-u)^rC(x-z)^{q-r}C_2(y-u)^{q-r}dx \ dy \ dz \ du \\
         &=\int_{(t_1D_1 \times t_2D_2)^4}  \quad  dx \ dy \ dz \ du \\
         &\hspace{2cm}   (K_1(x_1-y_1)+K_2(x_2-y_2))^r(K_1(z_1-u_1)+K_2(z_2-u_2))^r\\
         &\hspace{2cm}  \times (K_1(x_1-z_1)+K_2(x_2-z_2))^{q-r}(K_1(y_1-u_1)+K_2(y_2-u_2))^{q-r}
         \\
&=\sum_{k_1,k_2=0}^r\sum_{k_3,k_4=0}^{q-r}\binom{r}{k_1}\binom{r}{k_2}\binom{q-r}{k_3}\binom{q-r}{k_4} \\
         &\quad \int_{(t_1D_1)^4} dx_1 \ dy_1 \ dz_1 \ du_1\\
         &\hspace{3cm} K_1(x_1-y_1)^{k_1} K_1(z_1-u_1)^{k_2}K_1(x_1-z_1)^{k_3} K_1(y_1-u_1)^{k_4} \\
         &\quad \int_{(t_2D_2)^4} dx_2 \ dy_2 \ dz_2 \ du_2 \\
         & \hspace{3cm} K_2(x_2-y_2)^{r-k_1} K_2(z_2-u_2)^{r-k_2}K_2(x_2-z_2)^{q-r-k_3} K_2(y_2-u_2)^{q-r-k_4}\\
&=\sum_{k_1,k_2=0}^r\sum_{k_3,k_4=0}^{q-r}\mathbf{1}_{\{(k_1,k_2,k_3,k_4)\neq(0,0,0,0)\}}\binom{r}{k_1}\binom{r}{k_2}\binom{q-r}{k_3}\binom{q-r}{k_4} \\
         &\quad \int_{(t_1D_1)^4} dx_1 \ dy_1 \ dz_1 \ du_1 \\
         & \hspace{3cm} K_1(x_1-y_1)^{k_1} K_1(z_1-u_1)^{k_2}K_1(x_1-z_1)^{k_3} K_1(y_1-u_1)^{k_4} \\
         &\quad \int_{(t_2D_2)^4}  dx_2 \ dy_2 \ dz_2 \ du_2 \\
         & \hspace{3cm} K_2(x_2-y_2)^{r-k_1} K_2(z_2-u_2)^{r-k_2}K_2(x_2-z_2)^{q-r-k_3} K_2(y_2-u_2)^{q-r-k_4}\\
         &\quad +\vol(t_1D_1)^4\|f_2(t_2)\otimes_r f_2(t_2) \|^2\,.
    \end{align}
Moreover, we have
\begin{align}
& \sum_{k_1,k_2=0}^r \sum_{k_3,k_4=0}^{q-r} \mathbf{1}_{\{(k_1,k_2,k_3,k_4) \neq(0,0,0,0)\}} \binom{r}{k_1}\binom{r}{k_2}\binom{q-r}{k_3}\binom{q-r}{k_4} 
\\
& \quad \int_{(t_1D_1)^4} dx_1 \ dy_1 \ dz_1 \ du_1 \\
& \hspace{3cm} K_1(x_1-y_1)^{k_1} K_1(z_1-u_1)^{k_2}K_1(x_1-z_1)^{k_3} K_1(y_1-u_1)^{k_4}\\
&\quad \int_{(t_2D_2)^4}dx_2 \ dy_2 \ dz_2 \ du_2 \\
& \hspace{3cm} K_2(x_2-y_2)^{r-k_1} K_2(z_2-u_2)^{r-k_2}K_2(x_2-z_2)^{q-r-k_3} K_2(y_2-u_2)^{q-r-k_4} \\
    &  \lesssim \sum_{k_1,k_2=0}^r\sum_{k_3,k_4=0}^{q-r}\mathbf{1}_{\{(k_1,k_2,k_3,k_4)\neq(0,0,0,0)\}}\binom{r}{k_1}\binom{r}{k_2}\binom{q-r}{k_3}\binom{q-r}{k_4}  \times  \\
         & \hspace{1.3cm} \times 
        \int_{(t_1D_1)^2}K_1(x_1-y_1)^{k_1+k_3}dx_1dy_1\int_{(t_2D_2)^2}K_2(x_2-y_2)^{q-k_1-k_3}dx_2dy_2\\
        & \hspace{1.3cm} \times\int_{(t_1D_1)^2}K_1(z_1-u_1)^{k_2+k_4}dz_1du_1\int_{(t_2D_2)^2}K_2(z_2-u_2)^{q-k_2-k_4}dz_2du_2,
    \end{align}
where the last inequality follows from the positivity of $K_1,K_2$ and Lemma~\ref{lem:bound contr} applied to every term of the sum. Then, using Jensen as in the proof of Lemma \ref{lem:var_sum_2}, we obtain that, as $t_1,t_2\rightarrow\infty$,
\begin{align}
        & \|\widetilde f(t_1,t_2)\otimes_r  \widetilde f(t_1,t_2) \|_{2q-2r}^2  = (q!)^2 \frac{\|f(t_1,t_2)\otimes_r f(t_1,t_2) \|_{2q-2r}^2}{\var(Y(t_1,t_2))^2}\\
        & \quad \sim \frac{\|f(t_1,t_2)\otimes_r f(t_1,t_2) \|^2}{\vol(t_1D_1)^4 \var(A_2(t_2))^2}\\
        &\quad = \frac{\|f_2(t_2)\otimes_r f_2(t_2) \|^2}{\var(A_2(t_2))^2} \\
        & \hspace{1.3cm}+ O \left(\sum_{k_1,k_2=0}^r \sum_{k_3,k_4=0}^{q-r} \mathbf{1}_{\{(k_1,k_2,k_3,k_4)\neq(0,0,0,0)\}} \left(\frac{\gamma_{t_1}^1[q]}{\gamma_{t_2}^2[q]}\right)^{(k_1+k_3)/q}\left(\frac{\gamma_{t_1}^1[q]}{\gamma_{t_2}^2[q]}\right)^{(k_2+k_4)/q}\right)\\
        &\quad \sim (q!)^2 \frac{\|f_2(t_2)\otimes_r f_2(t_2) \|_{2q-2r}^2}{\var(A_2(t_2))^2} = \|\widetilde f_2(t_2)\otimes_r \widetilde f_2(t_2) \|_{2q-2r}^2.
    \end{align}
Therefore, the proof is again concluded by means of the Fourth Moment theorem \ref{thm:4th}.
\end{proof}

We conclude the section with an example.

\begin{example} \label{exa:slogt} 
Let us consider a Gaussian random field $B=(B_x)_{x\in\R^{d_1+d_2}}$ with covariance function of the form \eqref{equ:addsep}, 
choosing
\begin{equation}
    K_1(x_1)=\frac{1}{(1+\|x_1\|^2)^{\beta_1/2}}, \quad \quad K_2(x_2)=\frac{1}{(1+\|x_2\|^2)^{\beta_2/2}}.
\end{equation}
and consider its quadratic variation, that is \begin{equation}
    Y(t_1,t_2)[2]=\int_{t_1D_1\times t_2D_2}H_2(B_{x})dx.
\end{equation}
Suppose that $\beta_1\in (0,d_1/2)$and $\beta_2>d_2/2$. Note that $K_1$ satisfies both \eqref{equ:regvar} and \eqref{equ:regspec} (see \cite[Example 3]{LO13}). Then, by Theorem~\ref{thm:DM} the functional $\widetilde A_1(t_1)[2]$, see \eqref{marginal additive funct} properly normalized, is not asymptotically Gaussian; on the other hand, $\widetilde A_2(t_2)[2]$ is asymptotically Gaussian, thanks to Theorem~\ref{thm:BM}. Note that, up to constants, recalling Remark~\ref{rem:Quotients}, \begin{equation}
    \frac{\gamma^1_{t_1}[2]}{\gamma^2_{t_2}[2]} \sim  \frac{t_1^{-2\beta_1}}{t_2^{-d_2}}.
\end{equation}
Therefore, we may observe two different behaviors depending on the choice of the rate for $t_1$ and $t_2$: applying Theorem \ref{thm:main 4}, whenever $t_1^{-2\beta_1}=o( t_2^{-d_2})$, we have that $\widetilde{Y}(t_1,t_2)$ is not asymptotically Gaussian; conversely, when $ t_2^{-d_2}=o(t_1^{-2\beta_1})$ we have a Gaussian limiting behavior. 
\end{example}

\section*{Appendix}
\label{sec:appendix}

\begin{proof}[Proof of Lemma \ref{change}] The following proof is merely a reformulation of \cite[Theorem 4.5]{M81} with our notation. Since $h\in \mathcal{H}_{\nu}^{\odot q}$, we have the equality 
    \[
    h=\sum_{i_1,\dots,i_q=1}^\infty c_{i_1,\dots,i_q} \,{\rm Sym}(e_{i_1}\otimes \dots \otimes e_{i_q}),
    \]
    where $\{e_i\}_{i=1}^\infty$ is an orthonormal basis of $\mathcal{H}_{\nu}$ and Sym is the symmetrization operator. If we define $e_i'(x):=e_i(x)\, a(x)\in L_E^2(\nu')$, then $\{e'_i\}_{i=1}^\infty$ is an orthonormal basis of $\mathcal{H}_{\nu'}$ and we obtain 
    \[
    a^{\otimes q} h=\sum _{i_1,\dots,i_q=1}^\infty c_{i_1,\dots,i_q} \,{\rm Sym}(e'_{i_1}\otimes \dots \otimes e'_{i_q}).
    \]
    Applying the operators $I_{\nu,q}$ and $I_{\nu',q}$ we obtain 
    \[
    I_{\nu,q}(h)=\sum_{i_1,\dots,i_q=1}^\infty c_{i_1,\dots,i_q} \,I_{\nu,q}\left({\rm Sym}(e_{i_1}\otimes \dots \otimes e_{i_q})\right)
    \]
    and
    \[
    I_{\nu',q}(a^{\otimes q}\,h)=\sum_{i_1,\dots,i_q=1}^\infty c_{i_1,\dots,i_q} \,I_{\nu',q}\left({\rm Sym}(e'_{i_1}\otimes \dots \otimes e'_{i_q})\right).
    \]
    Therefore, we conclude by means of product formula (see e.g. \cite{bluebook}), observing that
    \[
    I_{\nu,q}(h)=\sum_{i_1,\dots,i_q=1}^\infty c_{i_1,\dots,i_q} X_{\nu}(e_{i_1}) \dots  X_{\nu}(e_{i_q})\overset{d}{=}\sum_{i_1,\dots,i_q=1}^\infty c_{i_1,\dots,i_q} N_{i_1} \dots  N_{i_q}
    \]
    and
    \[
    I_{\nu',q}(a^{\otimes q}h)=\sum_{i_1,\dots,i_q=1}^\infty c_{i_1,\dots,i_q} X_{\nu}(e'_{i_1}) \dots  X_{\nu}(e'_{i_q})\overset{d}{=}\sum_{i_1,\dots,i_q=1}^\infty c_{i_1,\dots,i_q} N_{i_1} \dots  N_{i_q},
    \]
    where $\{N_{i}\}_{i=1}^\infty$ is a sequence of i.i.d. standard Gaussian random variables.
\end{proof}

 \begin{proof}[Proof of Lemma \ref{change2}]
     Acting as in the proof of Lemma \ref{change}, we have
     \[
    I_{q}(h(\cdot))=\sum_{i_1,\dots,i_q=1}^\infty c_{i_1,\dots,i_q} \,X(e_{i_1}( \cdot)) \dots  X(e_{i_q}( \cdot))),
    \]
    where $\{e_i\}_{i=1}^\infty$ is an orthonormal basis of $\mathcal{H}_\lambda$. Note that using the multi-index notation $\underline s^{\alpha} = \prod_i s_i^{\alpha}$ for $s\in \mathbb N^d$ and $\alpha\in \R$, also the family $\{\underline{s}^{-1/2}e_i(\cdot/s_1,\dots,\cdot/s_d)\}_{i=1}^\infty$ is an orthonormal basis of $\mathcal{H}_\lambda$. Therefore, the two Gaussian families $(X(e_i(\cdot)))_{i=1}^\infty$ and $(X(\underline{s}^{-1/2}e_i(\cdot/s_1,\dots,\cdot/s_d)))_{i=1}^\infty$ share the same law, and, by extension, the following random variables are equally distributed: \begin{equation}
    I_{q}(h(\cdot))\os{\rm law}{=}\sum_{i_1,\dots,i_q=1}^\infty c_{i_1,\dots,i_q} \,X(\underline{s}^{-1/2}e_{i_1}( \cdot/s_{1},\dots,\cdot/s_d)) \dots \\ \dots  X(\underline{s}^{-1/2}e_{i_q}( \cdot/s_{1},\dots,\cdot/s_d)).
    \end{equation}
    Moreover, by product formula (see e.g. \cite{bluebook}), we have that
    \begin{multline}
        \sum_{i_1, \dots, i_q=1}^\infty  c_{i_1,\dots,i_q}  \,X(\underline{s}^{-1/2}e_{i_1}( \cdot/s_{1},\dots,\cdot/s_d)) \dots  X(\underline{s}^{-1/2}e_{i_q}(\cdot/s_{1},\dots,\cdot/s_d)) \\
    \overset{d}{=}I_{q}(\underline{s}^{-\frac{q}{2}}\tilde h(s_1,\ldots,s_d)).
    \end{multline}
    By linearity, $I_{q}(\underline{s}^{-\frac{q}{2}}\tilde h(s_1,\ldots,s_d))=\underline{s}^{-\frac{q}{2}}I_{q}(\tilde h(s_1,\ldots,s_d))$, which concludes the proof.
 \end{proof} 
 
\begin{lemma} \label{lem:bound contr} Let $K:\R^d\rightarrow \R$ be a non-negative definite function such that $K\ge0$. Suppose that $D\subseteq\R^d$ compact with ${\rm Vol}(D)>0$. Then, for every $k_1,k_2,k_3,k_4\in \mathbb N$, there exists a positive constant $c$, depending on $K, D$ and the exponents $k_i$, $i=1,2,3,4$, satisfying for all $t>0$ the following inequality: 
    \begin{multline}\label{equ:lemmabound}
        \int_{(tD)^4} K(x-y)^{k_1} K(z-u)^{k_2}K(x-z)^{k_3} K(y-u)^{k_4} dx  dy  dz  du \\
        \le c \int_{(tD)^2}K(x-y)^{k_1+k_3}dxdy \int_{(tD)^2}K(z-u)^{k_2+k_4}dzdu.
    \end{multline}
\end{lemma}

\begin{proof} 
    First, by using the inequality $x^ay^b\le x^{a+b}+y^{a+b}$, the LHS of the inequality \eqref{equ:lemmabound} can be bounded by 
\begin{multline}
    2 \int_{(tD)^4} K(x-y)^{k_1+k_3} K(z-u)^{k_2+k_4} dx  dy  dz  du \\ + 2 \int_{(tD)^4} K(x-y)^{k_1+k_3} K(y-u)^{k_2+k_4} dx  dy  dz  du.
\end{multline}    
Then, by the change of variable $a=x-y$ and $b=y-u$ and compactness of the domain $D$, we may also bound the second term in the previous equation, up to a positive constant, by
\begin{multline}
\vol(tD)\int_{(tD - tD)^2} K(a)^{k_1+k_3} K(b)^{k_2+k_4}  \vol(tD \cap (tD-a) \cap (tD-b))  da db \\ \le
    \vol(tD)^2 \int_{tD -  tD}K(a)^{k_1+k_3}da \int_{tD -   tD}K(b)^{k_2+k_4}db.
\end{multline}
To conclude, we need to show that \begin{equation}
    \vol (tD) \int_{tD - tD}K(a)^{k_1+k_3}da \lesssim \int_{(tD)^2} K(x-y)^{k_1+k_3}dxdy.
\end{equation}
By positivity of the integrand and the doubling conditions for non-negative definite functions that are also non-negative proved in \cite{Gorbachev}, we deduce that \begin{equation}
       \int_{tD - tD}K(a)^{k_1+k_3}da \le \int_{\{\|x\|\le t\cdot {\rm diam}(D)\} } K(a)^{k_1+k_3}da \lesssim \int_{\{\|x\|\le t \}}K(a)^{k_1+k_3}da.
\end{equation}
Reasoning as in \cite[Proof of Proposition 9, Step 1]{MN22}, we have that \begin{equation}\label{equ:remQuotients}
    t^d\int_{\{\|x\|\le t \}}K(a)^{k_1+k_3}da \asymp \int_{(tD)^2} K(x-y)^{k_1+k_3}dxdy,
\end{equation}
which is enough to conclude.
\end{proof}

\bibliographystyle{plain} 
\bibliography{pDOM} 

\end{document}